\documentclass[12pt]{amsart}
\usepackage{amsthm,amsfonts, amssymb, amscd}

\theoremstyle{plain}
\newtheorem{theorem}{Theorem}[section]
\newtheorem{lemma}[theorem]{Lemma}
\newtheorem{proposition}[theorem]{Proposition}
\newtheorem{corollary}[theorem]{Corollary}

\theoremstyle{definition}

\newtheorem{remark}[theorem]{Remark}
\newtheorem{conjecture}[theorem]{Conjecture}
\newtheorem{remarks}[theorem]{Remarks}

  \newcommand{\om}{\omega}    
  \newcommand{\sig}{\sigma}   
  \newcommand{\al}{\alpha}
  \newcommand{\del}{\delta}   \newcommand{\Del}{\Delta}
  \newcommand{\gam}{\gamma}   \newcommand{\Gam}{\Gamma}
  \newcommand{\lam}{\lambda} 
  \newcommand{\veps}{\varepsilon}
  \newcommand{\stab}{\operatorname{stab}}
   \newcommand{\eps}{\varepsilon}

  \def\b1{\text{\large 1}}  

  \def\ip<#1>{\langle#1\rangle}   

  \newcommand{\codim}{\operatorname{codim}}

\newcommand{\beqn}{\begin{equation}}
\newcommand{\eeqn}{\end{equation}}

\newcommand{\bc}{\mathbb{C}}
\newcommand{\bp}{\mathbb{P}}

\newcommand{\bz}{\mathbb{Z}}

 \newcommand{\ci}{\mathcal{I}}
 \newcommand{\cl}{\mathcal{L}}
 \newcommand{\co}{\mathcal{O}}
 \newcommand{\cs}{\mathcal{S}}

 \newcommand{\cf}{\mathcal{F}}

\begin{document}

 \numberwithin{theorem}{section}
  \title{On Positivity in $T$-equivariant $K$-theory
of Flag Varieties}
  \author{William Graham and Shrawan Kumar}

  \maketitle

 \begin{abstract} We prove some general results
on the $T$-equivariant $K$-theory $K_T(G/P)$ of the flag variety $G/P$, where $G$ is a semisimple
 complex
algebraic group, $P$ is a parabolic subgroup and $T$ is a maximal torus contained in $P$.
 In particular, we make a conjecture about a positivity
phenomenon in $K_T(G/P)$ for the product of two basis elements
written in terms of the basis of $K_T(G/P)$  given by the dual of the structure sheaf 
(of Schubert varieties) basis.
(For the full flag variety $G/B$, this dual basis is closely related
to the basis given by Kostant-Kumar.)
 This conjecture is parallel to (but different from) the conjecture of
Griffeth-Ram for the structure constants  of the product in the structure sheaf basis. 
We give
explicit expressions for the product in the $T$-equivariant $K$-theory of projective spaces
in terms of these bases.
In particular, we establish our conjecture and the conjecture of Griffeth-Ram in this case.
 \end{abstract}

  \section{Introduction}
Let $X$ denote the partial flag variety $G/P$, where $G$ is a complex semisimple
simply-connected algebraic
group and $P$ is a parabolic subgroup of $G$ containing a fixed Borel subgroup $B$.
The group $B$ acts
with finitely many orbits on $X$, and the closures of these orbits
(called the Schubert varieties) are indexed by $W^P$, the set  of minimal length
coset representatives
 of $W/W_P$ (where $W$ is the Weyl group of $G$ and $W_P$ is the Weyl group of $P$);
the Schubert variety corresponding to $w \in W^P$ is denoted $X^P_w$.
The Poincar\'e duals of the fundamental classes $[X_w^P]$ (called the Schubert
classes) form a basis
for the cohomology
ring $H^*(X)$.  The
structure constants of the multiplication in $H^*(X)$ with respect to this basis
have long been known to be non-negative.

This positivity result has
been generalized in different directions.  If $T$ is a maximal torus of
$B$, then the equivariant cohomology ring $H^*_T(X)$ is a free module
over $H^*_T(\text{pt})$, the equivariant cohomology ring of a point, again with a basis consisting
of Schubert classes.  As proved by Graham \cite{Gra:01}, the structure constants
in this basis again have a positivity property generalizing
the non-equivariant positivity.
Similarly, the Grothendieck group $K(X)$ has a basis consisting
of classes of structure sheaves $[\co_{X_w^P}]$ of Schubert varieties.  As proved by Brion
\cite{Bri:02}, the structure constants of the multiplication in
$K(X)$ have a predictable alternating sign behavior, which we will refer to
as a {\it positivity property}.

The positivity in $H^*_T(X)$ and
the positivity in $K(X)$ each imply the positivity in $H^*(X)$.  Our aim in this
paper is to discuss a positivity property for the multiplication in
the $T$-equivariant Grothendieck group $K_T(X)$ encompassing the positivity both in
$H_T(X)$ and $K(X)$. One subtlety in the $T$-equivariant $K$-theory is that
$K_T(X)$ has two natural quite different
bases: the basis consisting of classes of structure sheaves of Schubert
varieties (called the {\it structure sheaf} basis), and the dual basis with respect
to the natural pairing
on $K_T(X)$ (called the {\it dual structure sheaf} basis).

Surprisingly, both the structure sheaf basis and the dual structure sheaf basis of
$K_T(X)$ seem to exhibit the positivity phenomenon.   Let $R(T)$ denote
the representation ring of $T$, which   is a free abelian group with basis
consisting of the characters $e^{\lambda}$.
Let $\Delta$ denote the set of roots of  $\mbox{Lie }G$ with respect to $\mbox{Lie }T$,
and $\Delta^+$ the set of positive roots (chosen so that these are
the roots of $\mbox{Lie }B$).  Let $\{ \xi^w_P \}$ denote the dual basis to
the $R(T)$-basis $\{ [ \co_{X^P_w} ] \} $ of $K_T(X)$.
Write
$$
 [\co_{X^P_u} ]  [\co_{X^P_v} ] = \sum_{w\in W^P}
b^w_{u,v}(P)  [ \co_{X^P_w} ],
$$
and
 \[
\xi_P^u\, \xi_P^v = \sum_{w\in W^P} p^w_{u,v}(P) \xi_P^w,
  \]
 for (unique) elements $b^w_{u,v}(P)$ and $p^w_{u,v}(P)$ of $R(T)$.
 Griffeth and Ram conjectured a positivity property for the coefficients
 $b^w_{u,v}(P)$.  Specifically, their conjecture asserts that
 \beqn \label{e.intro1}
(-1)^{\dim (X)+\ell (u)+\ell (v)+\ell (w)} b^w_{u,v}(P) \in
\bz_+[e^{\beta}-1]_{\beta\in\Del^+}
  \eeqn
 (see Conjecture \ref{conj.GR} and Remark \ref{r.GR}). The validity of this conjecture
 for $P=B$ implies its validity for any $P$ (cf. Proposition \ref{p=b}).

In this paper we conjecture that the coefficients $p^w_{u,v}(P)$
  also exhibit the following positivity:
  \beqn \label{e.intro2}
(-1)^{\ell (u)+\ell (v)+\ell (w)}\, p^w_{u,v}(P) \in \bz_+[e^{-\beta}
-1]_{\beta\in\Del^+}
  \eeqn
  (see Conjecture \ref{conj.GK}).  It is not clear if the validity of the conjecture
  for $P=B$ implies that for any $P$.  On the other hand, this conjecture
  is compatible with the inclusion of flag varieties associated to Levi subgroups
  (see Proposition \ref{p.compatible}).
  Although the coefficients $b^w_{u,v}(P)$ and $ p^w_{u,v}(P)$ are
  related (see Propositions \ref{p.relation} and \ref{p.4.4}), it is not clear if
  one conjecture implies the other.

   The non-equivariant analogues
  of both these conjectures hold.  More precisely, if $F: R(T) \to \bz$
  is the forgetful map (sending each $e^{\lambda}$ to $1$), then
  $$
  (-1)^{\dim (X)+\ell (u)+\ell (v)+\ell (w)} F( b^w_{u,v}(P) )\ge 0
  $$
  and
  $$(-1)^{\ell (u)+\ell (v)+\ell (w)}\, F(p^w_{u,v}(P) )\ge 0.$$
   The first
  inequality is proved in \cite{Bri:02}.
  The second inequality can be easily deduced from
   \cite[Theorem 1]{Bri:02} (see Remark \ref{r.briondual}).
  In fact, we conjecture that an equivariant generalization
  of \cite[Theorem 1]{Bri:02} holds: Let $T'$ be  a subtorus of $T$.
  If $Y\subset X$ is a $T'$-stable
irreducible subvariety with rational singularities, and we write
  \[
[\co_Y] = \sum_{w\in W^P} a^Y_w [\co_{X^P_w}] ,
  \]
then
  \[
(-1)^{\codim Y +\codim X_w^P}\, a^Y_w\in\bz_+
[e^{-\beta}_{\mid T'} -1]_{\beta\in\Del^+}
\]
(see Conjecture \ref{conj.6.1}).  By Proposition \ref{p.pn.1},
 this would imply Conjecture
\ref{conj.GK}.

  The main purpose of this paper is to prove some results giving
  evidence for Conjectures \ref{conj.GK} and \ref{conj.6.1}.  The most substantial
  results are explicit formulas for the coefficients $p^w_{u,v}(P)$ and $b^w_{u,v}(P)$ in case
  $X = \bp^n$ (see Theorems \ref{t.pn.1} and \ref{t.pn.4}).  From this theorem we
  deduce recurrence relations (Theorems \ref{t.pn.2} and \ref{t.pn.5}) for the coefficients
  $p^w_{u,v}(P)$ and $b^w_{u,v}(P)$ which imply Conjectures \ref{conj.GK} and \ref{conj.GR}
  for $X = \bp^n$.   We also prove that in the case $P = B$,
  Conjecture \ref{conj.GK} holds for the coefficients $p^w_{u,e}$ and $p^w_{u, s}$,
  where $e$ and $s$ are (respectively) the identity element of $W$ and a simple reflection
  (see Proposition \ref{p.positive} and Remark \ref{r.positive}).
We verify Conjecture \ref{conj.6.1} in the case $Y$ is any opposite Schubert variety $X^w_P$
(cf. Proposition \ref{p.6.4} and Remark \ref{7.7}(a)).

A secondary purpose of this paper is to collect various results relating different
bases of $K_T(X)$, and relations among the structure constants
in these bases.   Among the natural bases of $K_T(X)$ are
the structure sheaf basis, the dual structure sheaf basis,
and the basis of the dualizing sheaves of Schubert
varieties.  Also, one can take opposite
Schubert varieties in place of Schubert varieties.
The positivity conjectures have different formulations in terms of these
different bases.
We describe some of the relations between these
bases and structure constants, in the hope that this paper will serve as a useful reference for other
workers in this area.

The contents of the paper are as follows.  Section 1 lays down the basic notation.
Section \ref{s.prelim}
contains some preliminary results on $K_T(G/P)$.  In particular, it identifies the
dual structure sheaf basis and also the basis of the dualizing sheaves of Schubert
varieties (cf. Propositions \ref{p.2.1} and \ref{p.bases}). Section \ref{s.conjecture}
contains the statement of our positivity conjecture (Conjecture \ref{conj.GK}).
We prove the conjecture for the coefficients $p^w_{u,e}$ and $p^w_{u,s}$ (for any
simple reflection $s$) in the case $P=B$ (cf. Proposition \ref{p.positive} and Remark
\ref{r.positive}). We observe that the conjecture has also been verified by an explicit calculation
for any rank $2$ group in the case $P=B$.
By a result of Brion, the nonequivariant analogue of this conjecture holds
(Remark \ref{r.briondual}).   This section also contains
the positivity conjecture of Griffeth and Ram (cf. Conjecture \ref{conj.GR}) and its equivalent
reformulation in terms of the dualizing sheaves (cf. Proposition \ref{3.12}).
It is shown that the validity of the Griffeth-Ram
conjecture for $P=B$ implies its validity for any $P$ (cf. Proposition \ref{p=b}).
Section \ref{s.relations} proves some relations between the structure
constants with respect to the structure sheaf basis and the
 dual structure sheaf basis (cf. Propositions \ref{p.relation} and \ref{p.4.4}).
 Section \ref{s.multiplicative}
proves that the structure constants with respect to either basis in the case $P=B$ lie
in the subring $\bz[e^{-\beta}-1]_{\beta\in\Del^+}$ of $R(T)$
(cf. Theorem \ref{t.structure} and Corollary \ref{5.2}).
Section \ref{s.pn} contains the explicit formula for the structure
constants in the case $X = \bp^n$ in  the
 dual structure sheaf basis, and the recurrence relation
implying Conjecture \ref{conj.GK} in this case (cf. Theorems \ref{t.pn.1},
\ref{t.pn.2} and \ref{t.pn.3}). Similar results are also obtained in the structure sheaf basis.
Section \ref{s.general} contains our more general conjecture
asserting the positivity of the coefficients of the class  of the structure sheaf of a $T'$-stable
subvariety $Y$ of $G/P$ with rational singularities written in terms of the structure sheaf
basis (cf. Conjecture \ref{conj.6.1}), for any subtorus $T'\subset T$.
We prove this conjecture in the special case where $Y$ is any
opposite Schubert variety in any $G/P$ (cf. Proposition \ref{p.6.4} and Remark \ref{7.7}(a)).

We thank M. Brion for some helpful conversations. The first author was
supported by the grant no. DMS-0403838 from NSF and the second author was supported by
the FRG grant no. DMS-0554247 from NSF.
\subsection{Definitions and notation}
We work with schemes over the ground field of complex numbers.

Let $X$ be a smooth algebraic variety with an action of
a torus $T$.  Let $K_T(X)$ denote the Grothendieck group of
$T$-equivariant coherent sheaves on $X$; because $X$ is
smooth, $K_T(X)$ may be identified with the Grothendieck
group of $T$-equivariant vector bundles on $X$.  Thus,
$K_T(X)$ is a ring; we will sometimes write the multiplication
in $K_T(X)$ using the notation of tensor product.
 The class in $K_T(X)$ of a $T$-equivariant coherent
sheaf $\cf$ will be denoted by $[\cf]$.  In particular, if
$Y \subset X$ is a $T$-stable closed subscheme, then
the structure sheaf of $Y$ defines a class $[\co_Y]$
in $K_T(X)$; if $Y $ is Cohen-Macaulay, then its
dualizing sheaf $\om_Y$ defines a class $[\om_Y]$ in $K_T(X)$.
 Let $*: K_T(X)\to
K_T(X)$ denote the standard involution taking a vector bundle to its dual and
$e^{\lam}$ to $e^{-\lam}$.  If $r \in R(T)$, we will sometimes write $\overline{r}$ for
$*r$, where $R(T)$ is the representation ring of $T$.   If $Y \supset Z$ are closed $T$-stable subschemes of $X$, then $\co_Y(-Z)$
is the ideal sheaf of $Z$ in $Y$.  Thus, viewed as an element of $K_T(X)$,
$[\co_Y(-Z)] = [\co_Y] - [\co_Z]$.

Recall that
$R(T)$  is a free abelian group (freely) generated  by the
characters $e^{\lambda}$ of $T$.  If $V$ is any representation
of $T$, we write $\text{ch }V$ for the corresponding element
of $R(T)$ (a linear combination of $e^{\lambda}$).
The group $K_T(X)$ is an $R(T)$-module.
If $X$ is proper, and $\cf$ is a $T$-equivariant coherent sheaf on
$X$, write $h^i(X,\cf) = \text{ch }H^i(X,F)$ and
 \[
\chi (X, \cf) := \sum_{p\geq 0} (-1)^p \text{ch } H^p(X, \cf) \in R(T).
  \]
 We extend this definition to define $\chi(X, \gamma)$ for any
$\gamma \in K_T(X)$.  We write
$\overline{h}^i(X, \cf) = * h^i(X,\cf)$ and
$\overline{\chi}(X, \gamma) = * \chi(X, \gamma)$.
For $X$ proper, there is
a pairing
 $$\ip< \cdot ,\cdot > : K_T(X) \otimes_{R(T)}
K_T(X)\to R(T)$$
given by
  \[
\ip< v_1,v_2> = \chi (X, v_1\otimes v_2).
  \]
  If $\cf$ is supported on a $T$-stable subscheme $Y$, then,
  viewing $\cf$ as a sheaf on $Y$, we have $\chi (X, \cf) = \chi (Y, \cf)$.

 Let $G$ be a semisimple connected simply-connected complex algebraic
group.  For the rest of the paper,  $T$ will denote a maximal
torus of $G$.
   Let $B$ be a Borel subgroup of $G$ containing $T$ and, as above,
let $\Delta$ denote the set of roots and
$\Delta^+$ the set of positive roots, chosen so that the roots
of $\text{Lie }B$ are positive.
Let $\{ \alpha_1, \ldots, \alpha_{\ell} \}\subset \Del^+$ denote the simple roots,
and let $s_i$ denote the simple reflection corresponding to $\alpha_i$.
Let $Q^+ := \sum_i\, \bz_+ \al_i $.
Let $\rho = \frac{1}{2} \sum_{\alpha \in \Delta^+} \alpha$. Let $B^-$ denote
the Borel subgroup of $G$
such that $B\cap B^- =T$.

 Let $P \supset B$ be a (standard) parabolic subgroup of $G$ and let $W_P$ be its Weyl
group.  Let $W^P$ be the set of the minimal length coset representatives
in $W/W_P$.  For $w\in W^P$, let $X^P_w$ (respectively, $X^w_P$) be the
Schubert variety (respectively, the opposite Schubert variety) defined by
  \begin{align*}
X^P_w &= \overline{BwP/P} \subset G/P\\
X^w_P &= \overline{B^-wP/P} \subset G/P.
  \end{align*}
(Here and elsewhere we use the same notation for elements of $W$ and
lifts of those elements to $G$.)  Set
  \[
\partial X_w^P = \bigsqcup_{\substack{v\in W^P\\ v<w}} BvP/P,
  \]
and
  \[
\partial X^w_P = \bigsqcup_{\substack{v\in W^P\\ v>w}} B^-vP/P.
  \]
  Except in Section \ref{s.pn},  we will abbreviate $X_w^B$ and $X^w_B$
  by $X_w$ and $X^w$, respectively.
  If $\lambda$ is a character of $P$ and $\bc_{\lambda}$ is the corresponding
  $1$-dimensional representation of $P$, let $\cl(\lambda)$ denote
  the line bundle $G \times_P \bc_{\lambda^{-1}}$ on $G/P$.

  Given an element $\gamma \in K_T(G/B)$, we write $\gamma(w)$ for the
  pullback of $\gamma$ to $K_T(\{wB\}) = R(T)$.

\section{Preliminary results on $K_T(G/P)$} \label{s.prelim}
Recall that $X^P_w$ and $X^w_P$ are Cohen-Macaulay (cf. \cite[Cor.~3.4.4]{BrKu:05})
and hence their dualizing sheaves $\om_{X^P_w}$ and $\om_{X^w_P}$
make sense.

It is well known that $\{ [ \co_{X^P_w} ] \}_{w\in W^P}$
is a $R(T)$-basis of $K_T(G/P)$, and so is $\{ [
\co_{X^w_P}] \}_{w\in W^P}$.  For any $w\in W^P$, set $\xi^w_P
= [\co_{X^w_P} (- \partial X^w_P)] \in K_T(G/P)$.

The next proposition is known and has been observed for example
by Knutson (see \cite[Section 8]{Buc:02}).

  \begin{proposition} \label{p.2.1}
For any $v,w\in W^P$,
  \[
\langle [\co_{X_w^P} ] , \xi^v_P \rangle = \del_{v,w} ,
  \]
i.e.,  $\{ [ \co_{X^P_w} ] \}_{w\in W^P}$ and $\{
\xi^w_P \}_{w\in W^P}$ are dual bases under the above pairing.
  \end{proposition}

  \begin{proof}
Since the intersections $X^P_w\cap X^v_P$ and $X_w^P\cap\partial X^v_P$
are proper ($\partial X^v_P$ is also Cohen-Macaulay since it is of pure
codimension 1 in the Cohen-Macaulay variety $X^v_P$), we get (by \cite[Lemma 1]{Bri:02})
  \[
\langle [\co_{X^P_w}], \xi^v_P \rangle = \chi ( G/P,
\co_{X^P_w\cap X^v_P} (-X^P_w\cap \partial X^v_P ) ) .
  \]
By \cite[Proposition 1]{BrLa:03},
  \begin{align*}
\chi ( \co_{X^P_w\cap X^v_P} ) &= 1 \; (\text{or }0)\\
\intertext{according as}
X^P_w\cap X^v_P &\neq \emptyset\; (\text{or }X^P_w\cap X^v_P =\emptyset ).
  \end{align*}
Similarly,
  \begin{align*}
\chi ( \co_{X^P_w\cap \partial X^v_P} ) &= 1 \; (\text{or }0)\\
\intertext{according as}
X^P_w\cap \partial X^v_P &\neq \emptyset\; (\text{or }X^P_w\cap
\partial X^v_P =\emptyset ).
  \end{align*}
Now,
  \begin{align*}
X^P_w\cap X^v_P\neq\emptyset &\Leftrightarrow w\geq v\quad\text{ and}\\
X^P_w\cap\partial X^v_P\neq\emptyset &\Leftrightarrow \text{there
exists a $\theta\in W^P$ such that } w\geq\theta >v\\
&\Leftrightarrow w>v.
  \end{align*}

Combining the above, we get the proposition.
  \end{proof}

Let $\{\tau^w\}_{w\in W}$ be the Kostant-Kumar $R(T)$-basis of $K_T(G/B)$
(cf. \cite[Remark 3.14]{KoKu:90}).  We abbreviate  $\xi^w_B$ by $\xi^w$
(as noted above, $X^B_w$ and $X^w_B$ are abbreviated as
$X_w$ and $X^w$,  respectively).

The next proposition gives some of the relations between various $T$-equivariant
sheaves on $G/B$ and between elements of $K_T(G/B)$.

  \begin{proposition}  \label{p.bases}
  For any $w\in W$
\begin{enumerate}
  \item[(a)] $\om_{X_w} \simeq e^{-\rho} \cl (-\rho )\otimes
\co_{X_w}(-\partial X_w)$ as $T$-equivariant sheaves.
  \item[(b)] $\om_{X^w} \simeq e^{\rho} \cl (-\rho )\otimes \co_{X^w}
  (-\partial X^w)$ as $T$-equivariant sheaves.
  \item[(c)] $*\tau^w = \xi^{w^{-1}} =  e^{-\rho} [\cl (\rho )] [\om_{X^{w^{-1}}} ]$, as elements of
  $K_T(G/B)$.
  \item[(d)] $e^{\rho} [\cl (\rho )](*\tau^w) = (-1)^{\ell (w)}
* [\co_{X^{w^{-1}}} ]$.
  \end{enumerate}
  \end{proposition}

  \begin{proof}  By \cite[Theorem 4.2]{Ram:87}, as non-equivariant sheaves,
$$\om_{X_w}\simeq \cl (-\rho )\otimes\co_{X_w}(-\partial X_w).$$  We now
determine $\om_{X_w}$ as a $T$-equivariant sheaf.  Since $BwB/B$ is a
smooth open subset of $X_w$, $\om_{X_w} |_{(BwB/B)}$ is the canonical line
bundle.  The fiber of $\om_{X_w}$ at the $T$-fixed point $wB\in BwB/B$ as
a $T$-module is given by the character $|\Del^-\cap
w\Del^+|= \sum_{\al\in\Del^-\cap w\Del^+} \al =w\rho -\rho$.
The fiber of $\cl (-\rho )$ at $wB$ has weight $w\rho$ and clearly the
fiber of $\co_{X_w}(-\partial X_w)$ at $wB$ has weight 0.  Combining the
above, we get (a).  (Here we have used the fact that on a reflexive sheaf
$\cs$ of rank 1 on an irreducible projective $T$-variety $X$, there exists
at most one $T$-equivariant structure such that the induced $T$-module
structure on the stalk of $\cs$ at a $T$-fixed point $x_0\in X$ is
trivial.)

The proof of (b) is similar.

By \cite[Proposition 3.39]{KoKu:90}, for any $v,w\in W$,
  \beqn
\chi ( X_{v^{-1}}, * \tau^w ) = \langle \co_{X_{v^{-1}}},
* \tau^w \rangle = \del_{v,w} .
  \eeqn
By the preceding proposition, this implies that $*\tau^w = \xi^{w^{-1}}$,
proving the first equality of (c).  The second equality of (c) follows
from (b) and the definition of $\xi^w$.  By
 \cite[\S2]{Bri:02} (which holds equivariantly),
 for any closed $T$-stable Cohen-Macaulay subvariety
 $Y \subset G/B$, we have
 $$
 *[\co_Y] = (-1)^{\codim Y} [\om_Y] \cdot *[\om_{G/B}].
 $$
Part (d) follows by combining (c) with this equation for $Y = X^{w^{-1}}$,
using the fact that $\om_{G/B} \cong \cl(-2 \rho)$.
  \end{proof}

  \section{A positivity conjecture for $K_T(G/P)$} \label{s.conjecture}
  \subsection{Positivity in the dual Schubert basis} \label{ss.conjecturedual}
  We make the following conjecture concerning the multiplication in
  $K_T(G/P)$ in terms of the basis $\{ \xi^w_P \}$.

    \begin{conjecture} \label{conj.GK}
For any (standard) parabolic subgroup $P$ and any $u,v\in W^P$, express
  \beqn \label{e.definep}
\xi^u_P\, \xi^v_P = \sum_{w\in W^P} p^w_{u,v}(P)\, \xi^w_P,
  \eeqn
for some (unique) $p^w_{u,v}(P) \in R(T)$.  Then,
  \[
(-1)^{\ell (u)+\ell (v)+\ell (w)}\, p^w_{u,v}(P) \in \bz_+[e^{-\beta}
-1]_{\beta\in\Del^+} ,
  \]
  where the notation $ \bz_+[e^{-\beta}
-1]_{\beta\in\Del^+}$ means polynomials in $\{e^{-\beta}
-1\}_{\beta\in\Del^+}$ with coefficients in $\bz_+$.
  \end{conjecture}

 We will write simply $p_{u,v}^w$ for $p_{u,v}^w(B)$.

 \begin{remarks} \label{r.conjecture}
  \begin{enumerate}
   \item[(a)] By an explicit case by case calculation, we have verified the
validity of the above conjecture for $P=B$ and any rank-2 group $G$.
  \item[(b)] We show the validity of our conjecture when $G=SL_{n+1}$ and
   $P$ is the standard maximal parabolic subgroup corresponding
to the first node (so that $G/P=\bp^n$) in Section \ref{s.pn}.
  \end{enumerate}
  \end{remarks}

The next proposition gives a relation between structure constants under the inclusion of 
flag varieties associated to Levi subgroups.  In the special case where the flag 
varieties are projective spaces, this result also follows from our explicit calculation 
of the structure constants (see Corollary \ref{c.pn.4}).

\begin{proposition} \label{p.compatible} Let $G,P,T$ be as in the above conjecture and 
let $L$ be the Levi subgroup of $P$ containing $T$. Let $Q$ be a standard parabolic 
subgroup of 
$G$ contained in $P$ and let $Q_L:=L\cap Q$ be the corresponding parabolic subgroup of 
$L$.
  Then, for any $u,v,w \in (W_P)^{Q_L}$,
\[p_{u,v}^w(Q_L)=p_{u,v}^w(Q),\]
where $p_{u,v}^w(Q_L)$ are the structure constants for the flag variety $L/Q_L$.
(Observe that  $(W_P)_{Q_L}$ can canonically be identified with $W_Q$ and $(W_P)^{Q_L}$
is canonically embedded in $W^Q$.)
\end{proposition}

\begin{proof} Observe that the canonical inclusion $i: L/Q_L\hookrightarrow G/Q$ takes 
the 
Schubert variety $X_w^{Q_L} \subset L/Q_L$ isomorphically onto the Schubert variety 
$X_w^{Q} \subset G/Q$, for any $w \in (W_P)^{Q_L}$. For $w \in W^{Q}$, we claim that 
$i^*(\xi^{w}_{Q})$ equals $\xi^{w}_{Q_L}$ if $w \in (W_P)^{Q_L}$, and is $0$ otherwise.  
Indeed, for $u \in (W_P)^{Q_L}$, 
$$ \chi(X_u^{Q_L}, i^*(\xi^{w}_{Q})) = \chi(X_u^{Q}, 
\xi^{w}_{Q} )
 = \delta_{u,w},
 $$
 proving the claim.

 For $u,v \in (W_P)^{Q_L}$, we have
$$
\xi^u_{Q_L}\, \xi^v_{Q_L} = \sum_{w\in (W_P)^{Q_L}} p^w_{u,v}(Q_L)\, \xi^w_{Q_L}.
$$
On the other hand, since $i^*$ is a  ring homomorphism,
we have
$$
\xi^u_{Q_L}\, \xi^v_{Q_L} = i^* (\xi^u_{Q}\, \xi^v_{Q}) =
i^* (\sum_{w \in W^{Q}} p^{w}_{u,v}(Q)\, \xi^w_{Q} ) = \sum_{w\in (W_P)^{Q_L}}
p^w_{u,v}(Q)\, \xi^w_{Q_L}.
$$
Comparing these two expressions, we get the proposition.
\end{proof}

  \begin{lemma}
  Let $\pi: G/B \to G/P$ denote the projection. Then
  \[
\pi^*(\xi^v_P) = \sum_{u\in vW_P}\xi^u, \text{  for any }v\in W^P.
  \]
  \end{lemma}

\begin{proof} We have
$\langle \pi^*(\xi^v_P), [\co_{X_u}] \rangle = \langle \xi^v_P, \pi_*[\co_{X_u}] \rangle$.
Further, $\pi_*[\co_{X_u}]=[\co_{\pi(X_u)}]$ by \cite[Theorem 3.3.4(a)]{BrKu:05}.
Thus, by Proposition \ref{p.2.1}, $\langle \xi^v_P, \pi_*[\co_{X_u}] \rangle$
 is $0$ unless $\pi(X_u)= X^P_v$, and this holds
if and only if $u \in v W_P$.  The lemma follows from this, together
with Proposition \ref{p.2.1} applied to the case of $G/B$.
\end{proof}

 Unlike Conjecture \ref{conj.GR} below due to Griffeth-Ram, the validity of the above conjecture
for $P=B$ does not seem to give the validity of the conjecture for an
arbitrary (standard) parabolic $P$.  In fact, we have the following proposition
relating the structure constants for $P$ and $B$.

\begin{proposition} For any $u,v,w \in W^P,$
  \[
p^w_{u,v}(P) = \sum_{\substack{u'\in uW_P\\ v'\in vW_P}} p^w_{u',v'}(B).
  \]
\end{proposition}

\begin{proof} Since
\[\xi^u_P\xi^v_P= \sum_{w\in W^P}\,p^w_{u,v}(P)\xi^w_P,\]
taking $\pi^*$ and using the above lemma, we get
\[\sum_{\substack{u'\in uW_P\\ v'\in vW_P}}\,\xi^{u'}\xi^{v'}= \sum_{w\in W^P}\,
\bigl(p^w_{u,v}(P)\sum_{w'\in wW_P}\xi^{w'}\bigr),\]
i.e.,
\[ \sum_{\theta\in W}\sum_{\substack{u'\in uW_P\\ v'\in vW_P}}\,p^\theta_{u',v'}(B)\xi^\theta =
\sum_{w\in W^P}
\sum_{w'\in wW_P}\,p^w_{u,v}(P)\xi^{w'}.\]
Equating the coefficients from the two sides, we get the proposition.
\end{proof}

  Let $D$ be the diagonal map $G/P
\to G/P\times G/P$.  This, of course, induces the push-forward map
  \[
D_* : K_T(G/P) \longrightarrow K_T(G/P)  \otimes_{R(T)} K_T(G/P),
  \]
  \[D_*[\cf]=\sum_{p\geq 0}\,(-1)^p [R^pD_*\cf],\]
and also the pull-back (product) map
  \[
D^*: K_T(G/P)  \otimes_{R(T)} K_T(G/P) \longrightarrow K_T(G/P).
  \]
  Here we have identified $K_T(G/P \times G/P)$ with $K_T(G/P)  \otimes_{R(T)} K_T(G/P)$
  (cf. \cite[Theorem 5.6.1]{ChGi:97}).
 The next proposition gives another description of the coefficients
  $p^w_{u,v}(P)$.

  \begin{proposition}  \label{p.pn.1} For any $u,v,w\in W^P$,
\[
D_* [\co_{X^P_w}] = \sum_{u,v \in W^P}\, p^w_{u,v}(P)[\co_{X^P_u}]\boxtimes
[\co_{X^P_v}] .
  \]
  \end{proposition}

  \begin{proof}  This follows from functorial properties of $K$-theory.
To see this, for any space $Y$, write $\pi_Y$ for the projection from
$Y$ to a point.  Write $X = G/P$.  By definition, $\chi (X,\cf) = \pi_{X*}(\cf)$.
By definition of the coefficients $p^w_{u,v}(P)$ and Proposition \ref{p.2.1},
  \begin{align*}
p^w_{u,v}(P) &= \pi_{X*}  (\xi_P^u \xi_P^v\otimes [\co_{X^P_w}] ) \\
&= \pi_{X*} (D^* (\xi_P^u\boxtimes\xi_P^v) \otimes
[\co_{X^P_w}] ) \\
&= (\pi_{X\times X})_* \, D_*  (D^* (\xi_P^u\boxtimes\xi_P^v)
\otimes [\co_{X^P_w}] ) \\
&= (\pi_{X\times X})_*  ( (\xi_P^u\boxtimes\xi_P^v)\otimes D_*
[\co_{X^P_w}] ) .
  \end{align*}
Since $\{\xi_P^u\boxtimes\xi_P^v\}$ and $\{ [\co_{X^P_u}]\boxtimes
[\co_{X^P_v}]\}$ are dual bases of $K_T(X\times X)$, the lemma follows.
  \end{proof}

\begin{remark} \label{r.briondual}
The non-equivariant analogue of the preceding proposition holds
with the same proof.  Combining this with \cite[Theorem 1]{Bri:02},
we see that the structure constants $F(p^w_{u,v}(P))$
for the non-equivariant multiplication in the
basis $\{\xi_P^u\}_u$ (cf. equation (4) of Conjecture \ref{conj.GK}); here
$F: R(T) \to \Bbb Z$ is the forgetful map) satisfy
\[(-1)^{\ell (w)+\ell (u)+\ell (v)}F(p^w_{u,v}(P)) \in \bz_+.\]
\end{remark}
\vskip2ex

For any subset $S\subset \{ 1,\cdots ,\ell\}$ (including $S=\emptyset$),
let $W_S$ be the subgroup of $W$ generated by the simple reflections $\{
s_i, i\in S\}$.  Recall that $Q^+ := \sum_i\, \bz_+ \al_i $.

  \begin{proposition} \label{p.positive} For any $u,w\in W$, and any
  $S \subset \{ 1,\cdots ,\ell\}$, we have
  \[
(-1)^{\ell (w)+\ell (u)}\sum_{v\in W_S} p^w_{u,v} \in \sum_{ \beta \in Q^+}
\bz_+ e^{-\beta}.
  \]
In particular,
  \[
(-1)^{\ell (w)+\ell (u)}\, p^w_{u,e} \in \bz_+[e^{-\beta}-1]_{\beta \in\Del^+} .
  \]
  \end{proposition}

  \begin{proof}  As in Proposition \ref{p.pn.1}, write
  \[
D_*[\co_{X_w}] = \sum_{u,v\in W} p^w_{u,v}[\co_{X_u}]\boxtimes
[\co_{X_v}].
  \]
Pairing this with $\xi^u\boxtimes \cl (-\rho_S)$, we get
  \beqn \label{e.positive1}
\chi  ( X_w, \xi^u\otimes \cl (-\rho_S) ) = \sum_{v\in W}
p^w_{u,v}\, \chi ( X_v, \cl (-\rho_S) ),
  \eeqn
where $\rho_i$ is the $i$-th
fundamental weight and  $\rho_S:=\sum_{i\notin S} \rho_i$.

We claim that, for any $v\notin W_S$,
  \beqn \label{e.positive2}
H^i ( X_v, \cl(-\rho_S) ) = 0 \text{  for all }i\geq 0.
  \eeqn

Let $P=P_S$ be the parabolic subgroup corresponding to the subset $S$,
i.e., the Levi subgroup of $P_S$ containing $T$ has for its simple roots $\{\al_i\}_{i\in
S}$.  Then, the line bundle $\cl (-\rho_S)$ is the pull-back of a line
bundle on $G/P$.  Moreover, by \cite[Theorem 3.3.4]{BrKu:05},
  \beqn \label{e.positive3}
H^i ( X_v, \cl (-\rho_S) ) \cong H^i ( X_{v'}, \cl
(-\rho_S) ),
  \eeqn
where $v'$ is the coset representative of minimal length in the coset
$vW_S$.  Take $s_j$, $j\notin S$, such that $v's_j<v'$.  Then, the standard
projection $\pi : X_{v'}\to X_{v'}^{P_j}$ is a $\bp^1$-fibration and $\cl
(-\rho_S)$ has degree -1 along the fibers of $\pi$, where $P_j=P_{\{j\}}$.  Hence,
$R^i \pi_*\cl (-\rho_S) = 0$ for all $i$;  \eqref{e.positive2} follows from this
and the Leray spectral sequence together with \eqref{e.positive3}.

For $v\in W_S$, by \eqref{e.positive3},
  \[
H^i ( X_v, \cl (-\rho_S) ) \cong H^i ( X_e, \cl
(-\rho_S) ).
  \]
Thus,
  \begin{align*}
H^i ( X_v, \cl (-\rho_S) ) &= 0 \quad\text{if $i>0$ and }\\
\text{ch } H^0 ( X_v, \cl (-\rho_S) ) &= e^{\rho_S}.
  \end{align*}
Thus, by \eqref{e.positive1},
\[\chi  ( X_w, \xi^u\otimes\cl (-\rho_S) )
= \Bigl(\sum_{v\in W_S} p^w_{u,v}\Bigr)\, e^{\rho_S},  \]
i.e.,
\beqn\label{(8)}
\sum_{v\in W_S} p^w_{u,v}
= e^{-\rho_S}\, \chi  ( X_w, \xi^u\otimes\cl (-\rho_S) ) .
  \eeqn
Now,
$$
\chi  ( X_w, \xi^u\otimes\cl (-\rho_S) ) = \chi  ( X_w\cap X^u, \cl (-\rho_S)  (-X_w\cap\partial X^u)).
$$
By \cite[Theorem 4]{Bri:02}, this equals
\beqn\label{(9)}
(-1)^{\ell (u)+\ell (w)}w_o\Bigl(*\text{ch}\Bigl( H^0  (
X_{w_ou}\cap X^{w_ow}, \cl (\rho_S)(-X_{w_ou}\cap
\partial X^{w_ow}) )\Bigr)\Bigr),
\eeqn
where $w_o$ is the longest element of $W$.
(Brion's result is stated non-equivariantly, but
if we change
his duality formula to the following:
\[c^w_v(\lambda)=(-1)^{\ell (v)+ \ell (w)} w_o\cdot(*c^{w_ov}_{w_ow}(-\lambda)),\]
then it remains true $T$-equivariantly by a similar proof.)
By \cite[Proposition 1]{BrLa:03}, the restriction map
  \[
H^0(G/B, \cl (\rho_S)) \longrightarrow H^0 ( X_{w_ou}\cap X^{w_ow},
\cl (\rho_S) )
  \]
is surjective.  Also,
$$
H^0  (X_{w_ou}\cap X^{w_ow}, \cl (\rho_S)(-X_{w_ou}\cap
\partial X^{w_ow}) ) \subset H^0 ( X_{w_ou}\cap X^{w_ow},
\cl (\rho_S) ).
$$
Since $H^0(G/B, \cl (\rho_S))$ is the irreducible $G$-module with highest weight
$- w_o \rho_S$, we see that
  \beqn\label{(10)}
e^{-\rho_S}\, w_o\Bigl(*\text{ch}\Bigl( H^0  (
X_{w_ou}\cap X^{w_ow}, \cl (\rho_S)(-X_{w_ou}\cap
\partial X^{w_ow}) )\Bigr)\Bigr) \in \sum_{\beta\in Q^+} \bz_+\, e^{-\beta} .
  \eeqn
Combining \eqref{(8)}-\eqref{(10)}, we get the lemma.
  \end{proof}

  \begin{remark}  \label{r.positive}
  Comparing the expression obtained in the above proof for
$p^w_{u,s_i}+p^w_{u,e}$ and $p^w_{u,e}$ (for any simple reflection $s_i$),
it can be shown that
  \[
(-1)^{\ell (u)+\ell (w)+1}\, p^w_{u,s_i} \in \bz_+\,
[e^{-\beta}-1]_{\beta\in\Del^+}.
  \]
  \end{remark}

\subsection{Positivity in the structure sheaf basis} \label{ss.conjectureGR} We
recall below the conjecture of
Griffeth and Ram on the nonnegativity of the product  in the structure sheaf basis.
They verified their conjecture for rank-2 groups by an explicit
case by case calculation.  In this section we prove that
the validity of their conjecture for $P=B$ implies its validity for
every (standard) parabolic subgroup $P$, and give
an equivalent formulation of their conjecture in terms of dualizing sheaves.

  \begin{conjecture}  \label{conj.GR}
  For any standard parabolic subgroup $P$ and $u,v\in
W^P$, express
  \beqn \label{e.GR1}
 [\co_{X_P^u} ]  [\co_{X^v_P} ] = \sum_{w\in W^P}
c^w_{u,v}(P)  [ \co_{X^w_P} ] \in K_T(G/P),
  \eeqn
for some (unique) $c^w_{u,v}(P)\in R(T)$.

Then,
  \[
(-1)^{\ell (u)+\ell (v)+\ell (w)} c^w_{u,v}(P) \in
\bz_+[e^{-\beta}-1]_{\beta\in\Del^+} .
  \]
  \end{conjecture}

  \begin{remark} \label{r.GR} Express
  \[
 [\co_{X^P_u} ]  [\co_{X_v^P} ] = \sum_{w\in W^P}
b^w_{u,v}(P)  [ \co_{X_w^P} ] .
  \]
Then, an equivalent formulation of the above conjecture asserts that
  \[
(-1)^{\dim (G/P)+\ell (u)+\ell (v)+\ell (w)} b^w_{u,v}(P) \in
\bz_+[e^{\beta}-1]_{\beta\in\Del^+} .
  \]
  Moreover, an argument similar to the proof of Proposition \ref{p.pn.1} shows
  that the structure constants
  $b_{u,v}^w(P)$
  are described by the equation:
  $$
  D_*(\xi^w_P) = \sum_{u,v \in W^P}\, b^w_{u,v}(P) \xi^u_P \boxtimes
\xi^v_P.
$$
\end{remark}

  \begin{proposition}\label{p=b}
We have
  \beqn \label{e.validity1}
c^w_{u,v}(P) = c^w_{u,v}(B), \text{ for any } u,v,w\in W^P.
  \eeqn
  Hence, the validity of the above Conjecture \ref{conj.GR} for $P=B$ implies its validity for
every (standard) parabolic subgroup $P$.
  \end{proposition}

  \begin{proof}  Let $\pi: G/B \to G/P$ be the standard projection.
Since $\pi$ is a $T$-equivariant smooth morphism, for any $T$-stable
closed subvariety $Z \subseteq G/P$,
  \begin{align} \label{e.validity2}
\pi^*[\co_Z] &= [\co_{\pi^{-1}(Z)} ] .\notag\\
\intertext{Thus,}
\pi^*[\co_{X^w_P}] &= [\co_{\pi^{-1}(X^w_P)}]
\notag\\
&= [\co_{X^w}] ,
  \end{align}
since $B^-wP/P = w_oBw_owP/P$ and $w_ow$ is the longest element in the
$W_P$-orbit $w_owW_P$.

Since $\pi^*$ is a ring homomorphism, \eqref{e.validity1} follows from
\eqref{e.validity2}.
  \end{proof}

As a consequence of this proposition, we will simply write $c_{u,v}^w$
for $c_{u,v}^w(P)$.

The following proposition provides an equivalent formulation of Conjecture \ref{conj.GR}
in terms of dualizing sheaves.

  \begin{proposition} \label{3.12} For any standard parabolic subgroup $P$ and any $u,v\in
W^P$, express
  \beqn \label{e.equiv1}
[\om_{X^u_P} ] \cdot [\om_{X^v_P}] = \sum_{w\in W^P}
d^w_{u,v}(P)[\om_{X^w_P} ] [\om_{G/P} ] ,
  \eeqn
for some (unique) $d^w_{u,v}(P)\in R(T)$.  Then,
  \beqn \label{e.equiv2}
d^w_{u,v}(P) = (-1)^{\ell (u)+\ell (v)+\ell (w)} *(c^w_{u,v}) .
  \eeqn
In particular, Conjecture \ref{conj.GR} is equivalent to the conjecture that
$d^w_{u,v}(P) \in \bz_+[e^{\beta}-1]_{\beta\in\Del^+}$.

Moreover,
\[d^w_{u,v}(P)=d^w_{u,v}(B),\,\,\text{ for any }\, u,v,w\in W^P.\]

  \end{proposition}

  \begin{proof}  By \cite[\S2]{Bri:02},
  \beqn \label{e.equiv3}
* [\co_{X^v_P} ] = (-1)^{\ell (v)} [\om_{X^v_P} ] \cdot
* [\om_{G/P}]  .
  \eeqn
Multiply equation \eqref{e.equiv1} by $  *[\om_{G/P}] ^2$ to get
  \[
 [\om_{X^u_P} ] *   [\om_{G/P}]   [\om_{X^v_P} ]
*   [\om_{G/P}]  = \sum_{w\in W^P} d^w_{u,v}(P)
 [\om_{X^w_P} ] *  [\om_{G/P}]  .
  \]
By \eqref{e.equiv3}, the above equation reduces to
  \[
* [\co_{X^u_P} ]\cdot\, *  [\co_{X^v_P} ]
= \sum_{w\in W^P}(-1)^{\ell (u)+\ell (v)+\ell (w)} d^w_{u,v}(P)
* [\co_{X^w_P} ] .
  \]
Comparing this with the identity \eqref{e.GR1} of Conjecture \ref{conj.GR}, we get
  \[
d^w_{u,v}(P) = (-1)^{\ell (u)+\ell (v)+\ell (w)} * (
c^w_{u,v}) .
  \]
This proves \eqref{e.equiv2}.
  \end{proof}

\section{Relations between structure sheaf constants} \label{s.relations}
In this section we restrict to the case $P=B$,
and prove two relations (Propositions \ref{p.relation}
and \ref{p.4.4})
between the structure constants
in the structure sheaf basis and the structure constants in
the dual basis.  However, we do not know how to use these
relations to relate the two positivity
conjectures \ref{conj.GK} and \ref{conj.GR}.

  Write
  \[
[\cl (\rho ) ] [\co_{X^{\theta}}] = \sum_{w\in W}\, d^{\theta}_w[\co_{X^w}],
  \]
for some (unique) $d^{\theta}_w\in R(T)$.
The following proposition gives a relation between the structure constants
$c^w_{u,v}$ and $p^w_{u,v}$.

  \begin{proposition}  \label{p.relation}
  For any $u,v,w\in W$,
  \[
c^w_{u,v} = (-1)^{\ell (u)+\ell (v)} \sum_{\theta\in W} (-1)^{\ell (\theta
)}\, e^{\rho}\, d^{\theta}_w\, \overline{p}^{\theta}_{u,v}.
  \]
  \end{proposition}

  \begin{proof}  By Proposition \ref{p.bases},
  \[
*\xi^w =  \tau^{w^{-1}} = (-1)^{\ell (w)}\, e^{\rho}\, [\cl (\rho )]
[\co_{X^w}].
  \]
  We have
  \[
(*\xi^u)(*\xi^v) = \sum_{\theta\in W} \overline{p}^{\theta}_{u,v}\,
(*\xi^{\theta}) ,
  \]
and hence
  \begin{align*}
[\co_{X^u}][\co_{X^v}] &= (-1)^{\ell (u)+\ell (v)}
\sum_{\theta\in W} (-1)^{\ell (\theta )}\, \overline{p}^{\theta}_{u,v}
e^{\rho}\, [\cl (\rho )] [\co_{X^{\theta}}] \\
 &= (-1)^{\ell (u)+\ell (v)} \sum_{\theta ,w\in W} (-1)^{\ell (\theta )}\,
\overline{p}^{\theta}_{u,v} e^{\rho} d^{\theta}_w [\co_{X^w}].
  \end{align*}
From this the proposition follows.
  \end{proof}

Before stating the second relation between structure
constants, we compare the bases $\{\xi^v\}_{v\in W}$ and $ \{
 [\co_{X^v} ] \}_{v\in W}$ of $K_T(G/B)$.  Let
 \[
  \mu(v,w) =
\begin{cases} (-1)^{\ell (v)+\ell (v)} &\text{if $v\leq w$}\\
0 &\text{otherwise}  \end{cases}
  \]
denote the M\"obius function of the Weyl group $W$.

  \begin{lemma} \label{l.4.3}
For any $v\in W$, write
  \begin{align*}
 [\co_{X^v} ] &= \sum_{w\in W} e_{v,w}\, \xi^w.  \\
\intertext{Then}
e_{v,w} &= 1, \quad\text{if } v\leq w \\
&= 0, \quad\text{otherwise}.
  \end{align*}
Thus,
  \[
\xi^v = \sum_w\, \mu(v,w) [\co_{X^w} ].
\]
  \end{lemma}

  \begin{proof}   By Proposition \ref{p.2.1},
  $$
  e_{v,w} =  \langle  [\co_{X_w} ] , [\co_{X^v} ] \rangle = \chi(G/B,  [\co_{X_w} ] \cdot [\co_{X^v} ] ).
  $$
  Since
 $X_w\cap X^v$ is a proper intersection,
 by \cite[Lemma 1]{Bri:02},
 $[\co_{X_w} ] \cdot [\co_{X^v} ] = [\co_{X_w\cap X^v} ]$.
 Thus, $e_{v,w} =  \chi  ( X_w\cap X^v, \co_{X_w\cap X^v} )$.
By \cite[Proposition 1]{BrLa:03} (cf. proof of Proposition \ref{p.2.1}),
 \[
\chi ( X_w\cap X^v, \co_{X_w\cap X^v} ) =1 \text{ (or 0)}
  \]
according as $X_w\cap X^v$ is nonempty (or empty), i.e.,
  \begin{align*}
\chi ( X_w\cap X^v,\co_{X_w\cap X^v} ) &= 1,\quad\text{if } v\leq w \\
&= 0 \quad\text{otherwise.}
  \end{align*}
This proves the first part of the lemma.

Define the matrix
  \[   E = (e_{v,w})_{v,w\in W} .  \]
Then, by \cite[\S3]{Deo:77}, $E^{-1}$ is the M\"obius function, i.e.,
  \[   ( E^{-1} )_{v,w} = \mu(v,w).  \]
From this the second part of the lemma follows.
  \end{proof}

We can now state the second relation between the structure constants.

  \begin{proposition} \label{p.4.4}
For any $u,v,w\in W$,
  \beqn \label{e.4.4.1}
c^w_{u,v} = (-1)^{\ell (w)}
\sum_{\substack{u\leq y\\ v\leq z\\ \theta\leq w}}
 (-1)^{\ell
(\theta )} p^{\theta}_{y,z} .
  \eeqn
 Similarly,
  \beqn \label{e.4.4.2}
p^w_{u,v} = (-1)^{\ell (u)+\ell (v)} \sum_{\substack{u\leq y\\ v\leq z\\
\theta\leq w}} (-1)^{\ell (y)+\ell (z)}\, c^{\theta}_{y,z}.
  \eeqn
  \end{proposition}

  \begin{proof}  By Lemma \ref{l.4.3},
  \begin{align*}
[\co_{X^u}]\, [\co_{X^v}] &=  (\sum_y\, e_{u,y}\, \xi^y )\cdot \Bigl( \sum_z \, e_{v,z}\, \xi^z\Bigr) \\
&= \sum_{y,z}\, e_{u,y} e_{v,z}\, \xi^y\xi^z \\
&= \sum_{y,z,\theta} e_{u,y}e_{v,z}\, p^{\theta}_{y,z}\xi^{\theta} \\
&= \sum_{y,z,\theta ,w} e_{u,y}e_{v,z}\, p^{\theta}_{y,z}\, \mu (\theta ,w)
[\co_{X^w}] \\
&= \sum_{\substack{u\leq y\\ v\leq z\\ \theta\leq w}} p^{\theta}_{y,z}\,
(-1)^{\ell (\theta )+\ell (w)} [\co_{X^w}] .
  \end{align*}
Thus, equating the coefficients in the $\{ [\co_{X^w}]\}_w$ basis, we get
  \[
c^w_{u,v} = (-1)^{\ell (w)} \sum_{\substack{ u\leq y\\ v\leq z\\
\theta\leq w}}\, (-1)^{\ell (\theta )} p^{\theta}_{y,z} .
  \]
To prove \eqref{e.4.4.2}, write by Lemma \ref{l.4.3},
  \begin{align*}
\xi^u \xi^v &= \Bigl( \sum_y\, \mu (u,y) [\co_{X^y}]\Bigr)
\cdot \Bigl( \sum_z\, \mu (v,z)[\co_{X^z}]\Bigr) \\
&= \sum_{y,z}\, \mu (u,y)\mu(v,z)[\co_{X^y}]\cdot [\co_{X^z}]\\
&= \sum_{y,z,\theta} \mu (u,y)\mu (v,z)\, c^{\theta}_{y,z} [\co_{X^\theta} ] \\
&= \sum_{y,z,\theta ,w} \mu (u,y) \mu (v,z)\, c^{\theta}_{y,z} e_{\theta ,w}
 \xi^w .
  \end{align*}
Thus,
  \begin{align*}
p^w_{u,v} &= \sum_{y,z,\theta} \mu (u,y)\mu (v,z)\, c^{\theta}_{y,z}
e_{\theta ,w} \\
&= \sum_{\substack{u\leq y\\ v\leq z \\ \theta \leq w}} (-1)^{\ell (u)+\ell (y)+\ell (v)+\ell (z)}\,
c^{\theta}_{y,z}.
  \end{align*}
This proves the proposition.
  \end{proof}

  \section{Multiplicative structure constants lie in $\bz
[e^{-\beta}-1]$}  \label{s.multiplicative}
 Let $Z$ denote the center of $G$, and let $T' = T/Z$.  The map $T \to T'$ induces
  an injection $R(T') \hookrightarrow R(T)$ whose image is
   the subring $\bz [e^{\beta}]_{\beta \in \Del}$ of $R(T)$.
  Of course,  $\bz [e^{\beta}]_{\beta \in \Del} = \bz [e^{\beta}-1]_{\beta \in \Del}$; writing
  the ring in this way emphasizes the relationship with the positivity conjectures.

The main result of this section is the following theorem concerning
the structure constants in the dual structure sheaf basis, in the case $P=B$.

  \begin{theorem}  \label{t.structure}
  With the notation as in Conjecture \ref{conj.GK}, for any $u,v,w\in W$, we have
  \[
p^w_{u,v} \in \bz [e^{-\beta}-1]_{\beta\in\Del^+}.
  \]
  \end{theorem}

  \begin{proof}
  We first show that $p_{u,v}^w \in R(T')$.
   Because $Z$ acts trivially on $G/B$, the action of $T$ on $G/B$ factors through
   the action of $T' = T/Z$.  Therefore, there is a canonical map $K_{T'}(X) \to K_T(X)$ compatible
   with the map $R(T') \to R(T)$.  By the cellular fibration lemma \cite[Lemma 5.5.1]{ChGi:97},
   $K_{T'}(X)$ is free over $R(T')$ and the
   classes of $\co_{X_w}$ in $K_{T'}(X)$ form a basis.  Since the class of $\co_{X_w}$
   in $K_{T'}(X)$ maps to the class of the same sheaf in $K_T(X)$, and the map
   $K_{T'}(X) \to K_T(X)$ is a ring homomorphism, the structure constants $b_{u,v}^w$of the multiplication
   in $K_T(X)$ with respect to this basis must be the images
   of the corresponding structure constants in $K_{T'}(X)$, and hence must lie
   in $R(T')$.  Thus, by \eqref{e.4.4.2},  $p_{u,v}^w$ must lie in $R(T')$ as well.

  By Proposition \ref{p.bases}(c),
  \beqn \label{e.structure1}
\xi^w = *\tau^{w^{-1}}, \text{  for any } w\in W.
  \eeqn
  By \cite[Proposition 2.22(h)]{KoKu:90}, $p^w_{u,v}=0$ unless $u,v\leq w$.
We will prove the proposition by induction on $\ell (w)$.
 The proposition is true for $w=1$ since  $p^1_{u,v}=0$ unless $u=v=1.$  Moreover,
 $p^1_{1,1} = 1$, since by
\cite[Proposition 2.22(b),(f)]{KoKu:90},
  \[
\tau^1 \tau^1 = \tau^1 + \sum_{w\neq e} d^w\tau^w,   \]
 for some $ d_w\in R(T)$.

Fix
$u,v\in W$ and assume by induction that $p^{\theta}_{u,v}\in\bz
[e^{-\beta}-1]$ for all $\theta <w$. Write
  \[
\tau^{u^{-1}}\cdot \tau^{v^{-1}} = \sum_{\theta <w}
a^{\theta}_{u,v}\tau^{\theta^{-1}} + a^w_{u,v}\tau^{w^{-1}} +
\sum_{\del\not\leq w} a^{\del}_{u,v}\tau^{\del^{-1}} .
  \]
  By \eqref{e.structure1},
    \beqn \label{e.structure2}
a^w_{u,v} = \overline{p}^w_{u,v}, \text{  for any } u,v,w\in W.
  \eeqn
Now,
  \beqn \label{e.structure3}
\tau^{u^{-1}}(w^{-1}) \tau^{v^{-1}}(w^{-1})= \sum_{\theta <w}
a^{\theta}_{u,v}\tau^{\theta^{-1}}(w^{-1}) + a^w_{u,v}
\tau^{w^{-1}}(w^{-1}),
  \eeqn
as $\tau^{\del^{-1}}(w^{-1})=0$ for $\del\not\leq w$ by \cite[Proposition 2.22(b)]{KoKu:90}.  By \cite{Wil:07} or \cite{Gra:02},
  \beqn \label{e.structure4}
  \tau^{\theta^{-1}}(w^{-1}) \in \bz [e^{\beta}-1]_{\beta\in\Del^+}.
    \eeqn
Moreover, by \cite[Proposition 2.22(b)]{KoKu:90},
  \beqn \label{e.structure5}
\tau^{w^{-1}}(w^{-1}) = \prod_{\nu\in w\Del^-\cap\Del^+} (1-e^{\nu}).
  \eeqn

  Let $x_i := e^{\al_i}$, $1\leq i\leq\ell$, where $\{\al_1,\dots
  ,\al_{\ell}\}$ are the simple roots.  Then, $R := \bz [e^{\beta}]_{\beta \in \Del^+}$
  is the polynomial ring $\bz [x_1,\dots ,x_{\ell}]$, and
    \[
R  \subset R(T') := \bz  [ x_1^{\pm
1},\dots ,x^{\pm 1}_{\ell} ] .
  \]
 We have proved earlier that $a^w_{u,v}$ is in $R(T')$; thus, to prove the theorem,
 it suffices by  \eqref{e.structure2}
to show that $a^w_{u,v} \in R$.
    By \eqref{e.structure2}--\eqref{e.structure4}, and induction,
  \beqn \label{e.structure6}
    a^w_{u,v}\, \tau^{w^{-1}}(w^{-1}) \in R.
    \eeqn
Moreover, by \eqref{e.structure5}, $\tau^{w^{-1}}(w^{-1})$ is in $R$ and does not
vanish at 0, so $a^w_{u,v}$ has no pole at 0.  Hence,
  \[   a^w_{u,v} \in \bc [x_1,\dots ,x_{\ell}] .   \]

  We next show that the coefficient of each monomial in $a^w_{u,v}$ must
be an integer, so that $a^w_{u,v}\in R$.  Write
  \[
a^w_{u,v} = \sum_{\underline{d}\in \bz_+^{\ell}}
c_{\underline{d}}\underline{x}^{\underline{d}}, \text{ for}\, c_{\underline{d}}
\in \Bbb C, \]
where
$\underline{x}^{\underline{d}} := x_1^{d_1}\cdots
x_{\ell}^{d_{\ell}}$ for $\underline{d} = (d_1,\dots ,d_{\ell}).$
Choose, if possible,  $c_{\underline{d}^o} \notin \bz$ so that $|\underline{d}^o|
:= d^o_1 + \cdots +d^o_{\ell}$ is minimum with this property.  Then, \eqref{e.structure6}
implies that
  \[
\Bigl( \sum_{|\underline{d}|\geq |\underline{d}^0|}
c_{\underline{d}} \, \underline{x}^{\underline{d}}\Bigr)\,
\tau^{w^{-1}}(w^{-1}) \in R,
  \]
which is a contradiction, since the constant term of
$\tau^{w^{-1}}(w^{-1})$ is $1$.  Hence, $a^w_{u,v} \in R$,
as desired.
  \end{proof}

  \begin{corollary} \label{5.2} For any $u,v,w \in W$, we have
 that $c^w_{u,v}\in\bz  [
e^{-\beta}-1 ]_{\beta\in\Del^+}$.
  \end{corollary}

\begin{proof}
This follows by combining Theorem \ref{t.structure} and Proposition
\ref{p.4.4}.  Alternatively, it can be deduced by an argument
similar to the proof of Theorem \ref{t.structure}, using the formula
of \cite{Gra:02}
for the pullback of elements of the structure sheaf basis to $T$-fixed
points.
\end{proof}

  \section{Positivity in equivariant $K$-theory of $\Bbb P^n$} \label{s.pn}
  In this section we prove an explicit formula for the structure constants
  in the dual structure sheaf basis in case $G/P = \bp^n$.  We use this to deduce a recurrence
  relation which implies Conjecture \ref{conj.GK} in this case.  We give the
  analogous results for the structure sheaf basis.
  This section can be read independently of the previous sections, except
  for references to a few results.

\subsection{Preliminary results on $K_T(\Bbb P^n)$}
Let $X=\Bbb P^n$ with projective coordinates $[x_1, \dots , x_{n+1}]$. Thus, $X=G/P$,
 where $G=SL_{n+1}$ and $P=\stab [1,0, \dots ,0]$.  Let $T=\{
(t_1, \dots , t_{n+1})\in (\bc^*)^{n+1} | \prod t_i=1\}$ be
 the maximal torus of $G$  acting on $X$ in the obvious way.
 Let $B$ denote the
set of upper triangular matrices in $G$.   Let $e^{\eps_i}\in\hat{T}$ be defined
by $e^{\eps_i}(t_1, \dots ,t_{n+1}) = t_i$, where $\hat{T}$ is the group of characters
of $T$. Written additively, we denote $e^{\eps_i}$ by $\eps_i$ itself.
 Then, the set of positive roots
is $\Del^+ = \{ \eps_i-\eps_j | 1\leq i < j\leq n+1\}$.
 Let $\chi_i =
\eps_1+\cdots +\eps_i, i=1, \dots , n+1$; for $i\leq n$ these are the
fundamental (dominant) weights.  The elements of $W^P$
can be identified with the set of integers $[n]:=\{0, 1, \ldots, n\}$.  In this section
we deviate from our convention in the rest of the paper and, for any $u\in [n]$,
simply write $X_u$, $X^u, \xi^u$ for $X_u^P, X^u_P$
and $\xi^u_P$ respectively.
Then, $X_u = \{ [x_1, \dots ,x_{u+1},0, \dots ,0]\}$ and
$X^u = \{ [0, \dots ,0,x_{u+1}, \dots , x_{n+1}]\}$.   Note that $\dim X_u =\codim
X^u =u$, and the intersections $X_u\cap X^u$ are transverse.
We will write simply $p_{u,v}^w$ for the structure
constants with respect to the basis $\{ \xi^u \}$ of $K_T(\Bbb P^n)$,
and $b_{u,v}^w$ for the structure
constants with respect to the basis $\{ [\co_{X_u} ] \}$.

Let
$V(\mu_1, \dots , \mu_k)$ denote the $T$-module with weights
$\mu_1, \dots ,\mu_k$.  Let $E_p(x_1, \dots ,x_r)$ denote the $p$-th elementary symmetric function in
the variables $x_1, \dots ,x_r$.  Recall that $D$ denotes the diagonal
embedding of $X$ in $X \times X$.

  \begin{lemma}  \label{l.pn.2}  If $\gam \in K_T(X)$, then, for any $u,w\in [n],$
  \begin{multline*}
\chi  ( X\times X, D_*[\co_{X_w}]\otimes ( [\co_{X^u}(-\partial
 X^u)] \boxtimes\gam ) )\\
= \chi (X_w\cap X^u, \, \gam ) - \chi (X_w\cap\partial X^u,\, \gam ).\qquad\qquad
  \end{multline*}
 \end{lemma}

  \begin{proof}  Arguing as in the proof of Proposition \ref{p.pn.1}, we get
  that the left hand side
equals $\chi(X,[\co_{X_w}]\otimes [\co_{X^u} (-\partial X^u)]\otimes\gam
)$.  This in turn equals
\[\chi (X, [\co_{X_w}]\otimes [\co_{X^u}]\otimes\gam
)-\chi (X, [\co_{X_w}]\otimes [\co_{\partial X^u}]\otimes\gam ).\]
  Since $X_w$ intersects $X^u$ and $\partial X^u$ properly and
these varieties are Cohen-Macaulay, Brion's result \cite[Lemma 1]{Bri:02} implies that
$[\co_{X_w}]\otimes [\co_{X^u}] = [\co_{X_w\cap X^u}]$ and
$[\co_{X_w}]\otimes [\co_{\partial X^u}] = [\co_{X_w\cap \partial X^u}]$.
Hence the above difference equals
  \[
\chi \Bigl( X,  [ \co_{X_w\cap X^u} ]\otimes\gam\Bigr) -
 \chi \Bigl( X,  [\co_{X_w\cap\partial X^u} ]\otimes\gam\Bigr) .
  \]
In general, if $Y$ is a closed subscheme of $X$,
the projection formula implies that $\chi (X,
[\co_Y]\otimes\gam ) = \chi (Y,\gam )$.  The result follows.
  \end{proof}

For any $n\in \Bbb Z$, the character $e^{n\eps_1}$ of $T$ extends to a character of $P$.
As earlier,  let $\bc_{n\eps_1}$ denote the corresponding $P$-module and
let $\cl (n\eps_1)$
denote the line bundle  $G \times_{P}\bc_{-n\eps_1}$ on $X=G/P=\Bbb P^n$.

  \begin{lemma} \label{l.pn.3} For any $n\in \Bbb Z$,
  $\cl (n\eps_1) \cong \co_X(n)$ as $G$-equivariant line bundles, where $\co_X(1)$
  denotes the dual of the tautological bundle.
  \end{lemma}

  \begin{proof}  Both $\cl (n\eps_1)$ and $\co_X(n)$ are sheaves of sections of
  $G$-equivariant line bundles.  Hence, the line bundles are determined by
  the character of $P$ on the fiber over the $P$-fixed point [1,0, \dots ,0].
  Since $P$ acts by the character $e^{-n\eps_1}$ on each line bundle, the line
  bundles are isomorphic.
  \end{proof}

  \begin{lemma}  \label{l.pn.4}
  If $w\geq u$, then as $T$-equivariant coherent sheaves on $X_w \cap X^u$,
  \[
\om_{X_w\cap X^u} = \cl  ( (-w+u-1)\,\eps_1 ) \big|_{_{X_w\cap X^u}} \otimes e^{\chi_u-\chi_{w+1}} .
  \]
  \end{lemma}

 \begin{proof} Observe first that $X_w\cap X^u=\{[0, \dots, 0, x_{u+1}, \dots, x_{w+1}, 0, \dots, 0]\}.$
 In particular, $X_w\cap X^u$ is a projective space. To prove the lemma,
we will show more generally that if $T$ acts on a vector space $V$ with
 weights $\mu_1, \dots ,\mu_{k+1}$, then, as $T$-equivariant sheaves on $\Bbb P(V)$, we have
  \[
\omega_{\Bbb P(V)} \cong e^{-(\mu_1 + \cdots + \mu_{k+1})} \,\co_{\Bbb P(V)}(-(k+1)),
  \]
  where $\Bbb P(V)$ denotes the space of lines in $V$.  This is a consequence of
  \cite[Ex.~III.8.4]{Har:77}; a direct proof is as follows.
We know that non-equivariantly $\omega_{\Bbb P(V)}\cong\co_{\Bbb P(V)}(-(k+1))$.
 So, equivariantly we must have $\omega_{\Bbb P(V)}\cong e^{\mu}\otimes\co_{\Bbb P(V)}(-(k+1))$,
 for some character $e^\mu$ of $T$.
 To determine $\mu$, observe that $T$ acts on $H^k(\Bbb P(V), \omega_{\Bbb P(V)})$
 by the trivial character $e^0=1$.  This holds since, by Serre duality, $H^k(\Bbb P(V),
 \omega_{\Bbb P(V)}) \cong H^0(\Bbb P(V),\, \co_{\Bbb P(V)})^*$. Moreover,
 this isomorphism is $T$-equivariant because the Serre duality is natural;
 since $T$ acts trivially on $H^0(\Bbb P(V),\, \co_{\Bbb P(V)})$, the claim follows.
 On the other hand, we claim that $T$ acts on $H^k ( \Bbb P(V)$, $ \co_{\Bbb P(V)}(-(k+1)) )$
 with weight $\mu_1 +\cdots +\mu_{k+1}$.
To prove this, observe that on the open set $U_i$ where the $i$-th coordinate
$x_i$ is nonzero, we have a section
  \[
\tau_i: [x_1, \dots ,x_{k+1}] \mapsto \Bigl(\frac{x_1}{x_i}, \frac{x_2}{x_i},
 \dots ,\frac{x_{k+1}}{x_i}\Bigr)
  \]
of $\co_{\Bbb P(V)}(-1)$, on which $T$ acts with weight $\mu_i$.  A generator of
$H^k(\Bbb P(V)$, $
\co_{\Bbb P(V)}(-(k+1)))$ is represented (as a \v{C}ech cocycle) by the
section $\tau_1\tau_2\cdots\tau_{k+1}\in$ $\Gam  ( U_1\cap\cdots\cap U_{k+1},\, \co (-(k+1) )$.
Since $T$ acts on this section with weight $\mu_1 +\cdots +\mu_{k+1}$, the claim follows.
 We conclude that, $T$-equivariantly,
  \[
\co_{\Bbb P(V)} (-(k+1))\cong \omega_{\Bbb P(V)}\otimes e^{\mu_1 + \cdots+\mu_{k+1}} ,
  \]
proving the lemma.
  \end{proof}

Let $Y_i\subset X$ be defined by the equation $x_i=0$.

  \begin{lemma} \label{l.pn.5} $[\co_{Y_i}] = 1-e^{-\eps_i} [\cl (-\eps_1)]$ in $K_T(X)$.
  \end{lemma}

  \begin{proof}  By Lemma \ref{l.pn.3}, $\cl(-\eps_1)=\co_X(-1)$.  Let $\ci_{Y_i}$
  denote the ideal sheaf of $Y_i$.  Since $[\co_{Y_i}]=1 -[\ci_{Y_i}]$,
  it suffices to show that $\ci_{Y_i}\cong e^{-\eps_i}\, \co_X(-1)$ as $T$-equivariant
  coherent sheaves on $X$.  On the open set $U_j: x_j\neq 0$, we have affine
  coordinates $\frac{x_k}{x_j}$ $(k\neq j)$.  Now, $\ci_{Y_i}(U_j)$ is generated by
  the section $\sig_j =\frac{x_i}{x_j}$, which transforms under $T$ by the weight
  $\eps_j-\eps_i$.  Similarly, $\co_X(-1)(U_j)$ is generated by the section
  \[
\tau_j: [x_1, \dots ,x_{n+1}] \mapsto \Bigl(\frac{x_1}{x_j}, \frac{x_2}{x_j},
 \dots , \frac{x_{n+1}}{x_j}\Bigr),
  \]
which transforms under $T$ by the weight $\eps_j$.  Since
  \[
\frac{\sig_k}{\sig_j} = \frac{\tau_k}{\tau_j} = \frac{x_j}{x_k},
  \]
it follows that the map $\co_X(-1)\otimes e^{-\eps_i}$ $\to \ci_{Y_i}$ defined
on $U_j$ by $\tau_j\otimes 1 \mapsto\sig_j$ is a $T$-equivariant sheaf isomorphism.
  \end{proof}

\begin{remark} \label{r.pn}
More generally, the following is true.  Let
$Y$ be any $T$-scheme and let $\cl$
be a $T$-equivariant line bundle on $Y$.  Given a section
 $\sigma$ of weight $\lambda$ and zero scheme $Z(\sigma)$, we have
 $$
 [\co_{Z(\sigma)}] = 1 - e^{\lambda}[\cl^*].
 $$
 See Proposition \ref{p.alpha} for a related result.
 \end{remark}

  \begin{corollary} \label{c.pn.1}
  \begin{align*}
\xi^v &=  e^{-\eps_{v+1}} [ \cl(-\eps_1)] \prod^v_{i=1}  ( 1-e^{-\eps_i}\, [\cl(-\eps_1)] )\,
 \qquad (0\leq v<n) \\
\xi^n &= \prod^n_{i=1}  ( 1-e^{-\eps_i}[\cl (-\eps_1)] ) .
  \end{align*}
  \end{corollary}

  \begin{proof} $\xi^v = [\co_{X^v}(-\partial X^v)] = [\co_{X^v}]
-[\co_{X^{v+1}}]$.  For $1\leq v\leq n$, since $X^v$ is the transverse
intersection of $Y_1, \dots , Y_v$, we have by Lemma \ref{l.pn.5},
  \[
[\co_{X^v}] = \prod^v_{i=1} [\co_{Y_i}] = \prod^v_{i=1} (1-e^{-\eps_i}\, [\cl (-\eps_1)]) .
  \]
(For $v=0$, this formula is interpreted as saying that $[\co_{X^0}]=1$,
which is true since $X^0=X$.)  A similar equation holds for $[\co_{X^{v+1}}]$
(with $[\co_{X^{n+1}}]=0$, as $X^{n+1}$ is empty).  Subtracting the two formulae gives the result.
  \end{proof}

\subsection{Structure constants with respect to the dual structure sheaf basis for $K_T(\Bbb P^n)$}
Write $[\sum_ia_it^i]_p = a_p$ and $[\sum_{i,j} b_{i,j}\, s^it^j]_{p,q}
= b_{p,q}$.  We have the following explicit formula for the structure constants
with respect to the dual structure sheaf basis.

  \begin{theorem} \label{t.pn.1} For any $0\leq u,v,w\leq n$,
  \begin{multline*}
(-1)^{u+v+w} p^w_{u,v}\\
 = e^{\chi_{w+1}-\chi_{u+1}-\chi_{v+1}} \biggl[
\frac{\bigl(\prod^u_{i=1} (1-te^{\eps_i})\bigr)\, \bigl(\prod^v_{i=1} (1-te^{\eps_i})\bigr)}
{\prod^{w+1}_{i=1} (1-te^{\eps_i})}\bigg]_{u+v-w+1} .\qquad\quad
  \end{multline*}
  \end{theorem}

  \begin{proof}  We have
  \begin{align}\label{(21)}
p^w_{u,v} &= \chi  ( X\times X,\, D_*[\co_{X_w}]\otimes
([ \co_{X^u}(-\partial X^u) ] \boxtimes\xi^v) ) \notag\\
&= \chi (X_w\cap X^u,\, \xi^v) - \chi (X_w\cap X^{u+1},\,\xi^v),
  \end{align}
where the first equality is by Propositions \ref{p.pn.1} and \ref{p.2.1} and the second is by Lemma
\ref{l.pn.2}.
Suppose first that $v<n$.  (We separate the case $v=n$ because the
formula for $\xi^v$ is different in this case.)  We have, by Corollary \ref{c.pn.1},
 \begin{align*}
\chi (X_w\cap X^u, \xi^v) &= e^{-\eps_{v+1}}\chi (X_w\cap X^u, [ \cl
(-\eps_1) ] \prod^v_{i=1} (1-e^{-\eps_i} [ \cl (-\eps_1) ] ) \\
&= e^{-\eps_{v+1}}\, \chi \Bigl( X_w\cap X^u, \sum^v_{p=0} \,(-1)^p E_p
(e^{-\eps_1}, \dots ,e^{-\eps_v})\, [ \cl (-(p+1)\eps_1) ] \Bigr)\\
&= \sum^v_{p=0} \,(-1)^pe^{-\eps_{v+1}}\, E_p(e^{-\eps_1}, \dots ,e^{-\eps_v})\, \chi
(X_w\cap X^u,\,[ \cl (-(p+1)\eps_1) ] ) .
  \end{align*}
We apply Serre duality to this formula, using the formula of Lemma \ref{l.pn.4} for
the dualizing sheaf.  Thus, the above sum can be expressed as  (note that
$X_w\cap
X^u $ is nonempty iff $w\geq u$ and $\dim X_w\cap
X^u $ $= w-u$):
  \[
(-1)^{w-u} \sum^v_{p=0}\,(-1)^p e^{-\eps_{v+1}}\,
E_p(e^{-\eps_1}, \dots ,e^{-\eps_v})\, \bar{\chi}(X_w\cap X^u, [ \cl
((p-w+u)\eps_1) ] \otimes e^{\chi_u-\chi_{w+1}}).
  \]
In this expression, the only contribution comes from the cohomology in
degree 0 (this follows from \cite[Theorem III.5.1]{Har:77}, since $X_w\cap X^u$ is a projective space). So,  the above expression reduces to
  \begin{align*}
(-1)^{w+u}e^{\chi_{w+1}-\chi_u-\eps_{v+1}} \sum^v_{p=0} &(-1)^p\, E_p
(e^{-\eps_1}, \dots ,e^{-\eps_v})\\
&\times \bar{h}^0 (X_w\cap X^u, \cl ((p-w+u)\,\eps_1))\\
=(-1)^{w+u}\, e^{\chi_{w+1}-\chi_u-\eps_{v+1}}\sum^v_{p=0} &(-1)^p\,
E_p(e^{-\eps_1}, \dots ,e^{-\eps_v})\\
&\times \text{ch}\Bigl(
S^{p-w+u} ( V(\eps_{u+1}, \dots ,\eps_{w+1}) )\Bigr).
  \end{align*}
Form the generating function
  \begin{align*}
f(s,t) &= \sum_{p,q} (-1)^p\, E_p(e^{-\eps_1}, \dots ,e^{-\eps_v})\,
\text{ch}\Bigl( S^q ( V(\eps_{u+1}, \dots ,\eps_{w+1}) )\Bigr)\, s^p t^q \\
&=  (\prod^v_{i=1} (1-se^{-\eps_i}) ) \, ( \prod^{w+1}_{j=u+1}
(1-te^{\eps_j})^{-1} ),
  \end{align*}
where the expansion of $(1-te^{\eps_j})^{-1}$ is at $t=0$ (i.e., as a
power series in $t$).  Then,
  \[
\chi (X_w\cap X^u, \xi^v) = (-1)^{w+u}\,
e^{\chi_{w+1}-\chi_u-\eps_{v+1}} \sum_p\, [f(s,t)]_{p,p-w+u}.
  \]
Setting $s=t^{-1}$, we obtain
  \begin{align*}
\chi &( X_w\cap X^u, \xi^v)
= (-1)^{w+u}\, e^{\chi_{w+1}-\chi_u-\eps_{v+1}} [f(t^{-1},t)]_{u-w}\\
&= (-1)^{w+u}\, e^{\chi_{w+1}-\chi_u-\eps_{v+1}} \Biggl[t^v (\prod^v_{i=1}
(1-t^{-1}e^{-\eps_i}) )\,  (\prod^{w+1}_{j=u+1}
(1-te^{\eps_j})^{-1} )\Biggr]_{u+v-w}\\
&= (-1)^{u+v+w}\, e^{\chi_{w+1}-\chi_u-\eps_{v+1}} \Biggl[ (\prod^v_{i=1}
(e^{-\eps_i}-t) )\, (\prod^{w+1}_{j=u+1}
(1-te^{\eps_j})^{-1} )\Biggr]_{u+v-w}\\
&= (-1)^{u+v+w}\, e^{\chi_{w+1}-\chi_u-\chi_{v+1}} \Biggl[ (\prod^v_{i=1}
(1-te^{\eps_i}) )\, (\prod^{w+1}_{j=u+1}
(1-te^{\eps_j})^{-1} )\Biggr]_{u+v-w},
  \end{align*}
where multiplying by
$t^v$ enabled us to shift degree from $u-w$ to $u+v-w$,
and in the last step we have used the equation
$\chi_{v+1}=\eps_1+...+\eps_{v+1}$.     From this last expression, we obtain
  \begin{multline*}
\chi (X_w\cap X^u, \xi^v)\\
= (-1)^{u+v+w}\, e^{\chi_{w+1}-\chi_u-\chi_{v+1}} \Biggl[
\frac{\bigl(\prod^u_{i=1}(1-te^{\eps_i})\bigr)\bigl(\prod^v_{i=1}(1-te^{\eps_i})\bigr)}
{\prod^{w+1}_{i=1}(1-te^{\eps_i})}\Biggr]_{u+v-w} .
  \end{multline*}
Set
\[A=\frac{\bigl(\prod^u_{i=1}(1-te^{\eps_i})\bigr)\bigl(\prod^v_{i=1}(1-te^{\eps_i})\bigr)}
{\prod^{w+1}_{i=1}(1-te^{\eps_i})}.\]
 Replacing $u$ by $u+1$ in
the formula for $\chi (X_w\cap X^u, \xi^v)$ yields an expression for $\chi (X_w\cap X^{u+1},\xi^v)$
(the expression turns out to be $0$ if $u=n$, so we can allow the case $u=n$ in the argument).
The difference is
  \begin{align}\label{(22)}
\chi &(X_w\cap X^u, \xi^v) - \chi (X_w\cap X^{u+1}, \xi^v)\notag\\
&= (-1)^{u+v+w}\, e^{\chi_{w+1}-\chi_u-\chi_{v+1}} [A]_{u+v-w}\notag\\
&\qquad\quad - (-1)^{u+v+w+1}\, e^{\chi_{w+1}-\chi_{u+1}-\chi_{v+1}}
[A(1-te^{\eps_{u+1}})]_{u+v-w+1} \notag\\
&= (-1)^{u+v+w}\, e^{\chi_{w+1}-\chi_u-\chi_{v+1}} \Bigl( [tA]_{u+v-w+1}
+ e^{-\eps_{u+1}}[A(1-te^{\eps_{u+1}})]_{u+v-w+1}\Bigr) \notag\\
 &= (-1)^{u+v+w}\, e^{\chi_{w+1}-\chi_u-\chi_{v+1}}  [
A(t+e^{-\eps_{u+1}}(1-te^{\eps_{u+1}}) ]_{u+v-w+1}\notag\\
&= (-1)^{u+v+w}\, e^{\chi_{w+1}-\chi_u-\eps_{u+1}-\chi_{v+1}}
[A]_{u+v-w+1}.
  \end{align}
Since $\chi_u+\eps_{u+1} = \chi_{u+1}$, we obtain the desired
expression from \eqref{(21)}.  This proves the result if $v<n$.  Interchanging the roles
of $u$ and $v$ proves the result if $u<n$.  The only remaining case is
$u=v=n$.  Rather than imitate the argument above for this case, we will
simply calculate both sides of the equation in the statement of the
theorem.  Since $\xi^n =[\co_{X^n}]$ and $X^n = \{ [0,\ldots,0,1]\}$, the
self-intersection formula (cf. \cite[Theorem 2.1]{vistoli}) implies that
  \[
\xi^n\xi^n =  (\prod^n_{i=1} (1-e^{\eps_{n+1}-\eps_i}) )\, \xi^n .
  \]
So
\beqn\label{(23)} p^w_{n,n}=0\,\,\text{ if} \,w<n\,\text{ and}
\eeqn
  \beqn\label{(24)}
(-1)^n\, p^n_{n,n} = (-1)^n\prod^n_{i=1} (1-e^{\eps_{n+1}-\eps_i}) =
\prod^n_{i=1}(e^{\eps_{n+1}-\eps_i}-1).
  \eeqn
Observe that the right  side of the equation in the statement of the
theorem is 0 unless $w=n$ (see the proof of Corollary \ref{c.pn.2} below).  So, it
suffices to calculate this side if $w=n$.  The case $n=1$ gives the correct
result (we omit the calculation).  Assume the result holds for $n-1$.
We have
  \begin{align*}
e^{-\chi_{n+1}} &\Biggl[
\frac{\prod^n_{i=1}(1-te^{\eps_i})}{1-te^{\eps_{n+1}}}
\Biggr]_{n+1}\\
&= e^{-\chi_{n+1}}\Biggl[\frac{\prod^n_{i=1}(1-te^{\eps_i})}
{1-te^{\eps_{n+1}}} -
\frac{(1-te^{\eps_{n+1}})\cdot \prod^n_{i=2}(1-te^{\eps_i})}{1-te^{\eps_{n+1}}}\\
&\qquad\qquad\quad +
\frac{(1-te^{\eps_{n+1}})\cdot \prod^n_{i=2}(1-te^{\eps_i})}
{1-te^{\eps_{n+1}}}\Biggr]_{n+1}\\
&= e^{-\chi_{n+1}}\Biggl[\frac{ t(e^{\eps_{n+1}}-e^{\eps_1})\cdot \prod^n_{i=2}(1-te^{\eps_i})
}{1-te^{\eps_{n+1}}}
 + \prod^n_{i=2}(1-te^{\eps_i})\Biggr]_{n+1}.
  \end{align*}
The second product is of degree $n-1$ in $t$ and therefore contributes
nothing to the degree $n+1$ part.  So, the above expression equals
  \begin{align*}
e^{-\chi_{n+1}} & ( e^{\eps_{n+1}} -e^{\eps_1} )
\Biggl[\frac{\prod^n_{i=2}(1-te^{\eps_i})}{1-te^{\eps_{n+1}}}\Biggr]_n\\
&=  ( e^{\eps_{n+1}-\eps_1}-1 ) e^{-\chi_{n+1} + \eps_1}
\Biggl[\frac{\prod^n_{i=2}(1-te^{\eps_i})}{1-te^{\eps_{n+1}}}\Biggr]_n\\
&=  ( e^{\eps_{n+1}-\eps_1}-1 )
\prod^n_{i=2} ( e^{\eps_{n+1}-\eps_i}-1 ),
  \end{align*}
where in the last step we have used our inductive hypothesis.  The
result follows from identity \eqref{(24)}.
  \end{proof}

  \begin{corollary} \label{c.pn.2} If $p^w_{u,v}\neq 0$, then $u,v\leq w\leq u+v+1$.
  \end{corollary}

  \begin{proof} Keeping the notation of the previous proof, we observe
that the power series expansion of $A$ contains no negative powers of
$t$.  Hence, if $p^w_{u,v} \neq 0$ then, by the previous theorem,  $u+v-w+1\geq 0$, i.e.,
$w\leq
u+v+1$.  Next, if $w\leq u-1$, then $A$ is a polynomial of degree
$u+v-w-1$.  Hence, $[A]_{u+v-w+1}=0$, so $p^w_{u,v}=0$.  Thus, if
$p^w_{u,v}\neq 0$ then $u\leq w$; similarly, $v\leq w$.
  \end{proof}

  \begin{corollary}  \label{c.pn.3}
  \begin{align*}
 p^w_{u,0} &= 0\quad\text{ if }\quad u\neq w-1,w,\\
 p^w_{w-1,0} &= -e^{\eps_{w+1}-\eps_1}, \,\text{ for } \, w\geq 1\\
 p^w_{w,0} &= e^{\eps_{w+1}-\eps_1}.
  \end{align*}
  \end{corollary}

  \begin{proof}  The first statement is immediate from Corollary \ref{c.pn.2}.  The
other formulas are easy consequences of Theorem \ref{t.pn.1}.
  \end{proof}

  \begin{corollary}  \label{c.pn.4} If $n,m\geq w$ then the structure constants
$p^w_{u,v}$ are the same for $\Bbb P^n$ and $\Bbb P^m$.
  \end{corollary}

  \begin{proof}  This follows immediately from the expression for
$p^w_{u,v}$ given by the theorem.
  \end{proof}

If $w\leq n$ and $\mu_1, \dots , \mu_{n+1}$ are weights of $T$, let
$p^w_{u,v}(\mu_1, \dots , \mu_{n+1})$ denote the element of $R(T)$ obtained by
replacing $\eps_i$ by $\mu_i$ in the expression of the theorem for
$p^w_{u,v}$, for $i=1, \dots , n+1$.

Define $\tilde{p}^w_{u,v}=(-1)^{u+v+w} p^w_{u,v}$.  The next theorem
gives a recurrence for the $\tilde{p}^w_{u,v}$.  We set
$\tilde{p}^w_{u,v}=0$ if $u$, $v$ or $w$ is negative.

  \begin{theorem} \label{t.pn.2} If $v\geq 1$, then
  \[
\tilde{p}^w_{u,v} =  ( e^{\eps_{u+1}-\eps_1}-1 )\,
\tilde{p}^{w-1}_{u-1,v-1}(\eps_2, \dots ,\eps_{n+1})
+ e^{\eps_{u+2}-\eps_1}\, \tilde{p}^{w-1}_{u,v-1}(\eps_2, \dots ,
\eps_{n+1}).
  \]
  \end{theorem}

  \begin{proof}  Let $A$ be as in the proof of Theorem \ref{t.pn.1}.  We can cancel
common factors from the numerator and denominator to write
  \[
A = \frac{\prod^v_{i=1}  (
1-te^{\eps_i} )}{\prod^{w+1}_{j=u+1} ( 1-te^{\eps_j})} .
  \]
Let
  \[
B = \frac{ ( 1-te^{\eps_{u+1}} ) \prod^v_{i=2}  (
1-te^{\eps_i} )}{\prod^{w+1}_{j=u+1} ( 1-te^{\eps_j})}.
  \]
By Theorem \ref{t.pn.1},
  \begin{align*}
\tilde{p}^w_{u,v} &=
e^{\chi_{w+1}-\chi_{u+1}-\chi_{v+1}} [A]_{u+v+1-w}\\
&=  e^{\chi_{w+1}-\chi_{u+1}-\chi_{v+1}} [A-B+B]_{u+v+1-w}.
  \end{align*}
We have
  \begin{multline*}
e^{\chi_{w+1}-\chi_{u+1}-\chi_{v+1}} [A-B]_{u+v+1-w}\\
= e^{\chi_{w+1}-\chi_{u+1}-\chi_{v+1}}\Biggl[ \frac{t (
e^{\eps_{u+1}}-e^{\eps_1} ) \prod^v_{i=2} ( 1-te^{\eps_i} )}
{\prod^{w+1}_{j=u+1}  ( 1-te^{\eps_j} )}\Biggr]_{u+v+1-w}.\qquad
  \end{multline*}
We can get rid of the factor $t$ in the numerator by taking the part
in degree $u+v-w$.  Since $e^{\eps_{u+1}} -e^{\eps_1}$ $=
e^{\eps_1} ( e^{\eps_{u+1}-\eps_1}-1 )$, we see that the
above expression equals
  \[
 ( e^{\eps_{u+1}-\eps_1}-1 )\,
 e^{\chi_{w+1}+\eps_1-\chi_{u+1}-\chi_{v+1}} \Biggl[
\frac{\prod^v_{i=2} ( 1-te^{\eps_i} )}
{\prod^{w+1}_{j=u+1}  ( 1-te^{\eps_j} )}\Biggr]_{u+v-w}.
  \]
Using the definition $\chi_k = \eps_1 + ... + \eps_k$, we see that the
above expression equals
  \[
 ( e^{\eps_{u+1}-\eps_1} -1 )\,
\tilde{p}^{w-1}_{u-1,v-1}(\eps_2, \dots , \eps_{n+1}).
  \]
Next, we have
  \begin{multline}\label{(25)}
\qquad\quad e^{\chi_{w+1}-\chi_{u+1}-\chi_{v+1}}[B]_{u+v+1-w}\\
= e^{\chi_{w+1}-\chi_{u+1}-\chi_{v+1}}
\Biggl[ \frac{\prod^v_{i=2} ( 1-te^{\eps_i} )}
{\prod^{w+1}_{j=u+2}  ( 1-te^{\eps_j} )}\Biggr]_{u+v+1-w}.\qquad\quad
  \end{multline}
Now,
  \begin{align*}
\chi_{w+1}-\chi_{u+1}-\chi_{v+1} = (\eps_1+...+\eps_{w+1}) &-
(\eps_1+...+\eps_{u+1})\\
 &- (\eps_1+...+\eps_{v+1})\\
= (\eps_2+...+\eps_{w+1}) &- (\eps_2+...+\eps_{u+2})\\
&- (\eps_2+...+\eps_{v+1}) + (\eps_{u+2}-\eps_1).
  \end{align*}
Using this we see that the expression \eqref{(25)} equals
  \[
e^{\eps_{u+2}-\eps_{1}}\, \tilde{p}^{w-1}_{u,v-1} (\eps_2, \dots ,\eps_{n+1}).
  \]
This proves the theorem.
  \end{proof}

As an immediate consequence of these results, we can verify Conjecture \ref{conj.GK}
for projective space:

  \begin{theorem}  \label{t.pn.3} For all $0\leq u,v,w\leq n$,
  \[
\tilde{p}^w_{u,v} \in \bz_+ [ e^{-\al}-1 ]_{\al\in\Del^+}.
  \]
  \end{theorem}

  \begin{proof}  This holds if $v=0$ by Corollary \ref{c.pn.3}.  The general case
follows by induction using the recurrence of Theorem \ref{t.pn.2}.
  \end{proof}

  \begin{remark}  (a) For $0\leq u,v,w\leq n$, define
  \[
\tilde{q}^w_{u,v} = (-1)^{u+v+w}\, \chi (X_w\cap X^u, \xi^v).
  \]
Then, by the same proof as that of Theorems \ref{t.pn.1}
and \ref{t.pn.2}, we have the recursion
  \[
\tilde{q}^w_{u,v} =  ( e^{\veps_{u+1}-\veps_1}-1 )\,
\tilde{q}^{w-1}_{u-1,v-1}(\veps_2,\dots ,\veps_{n+1}) +
e^{\veps_{u+1}-\veps_1}\, \tilde{q}^{w-1}_{u,v-1}(\veps_2,\dots ,\veps_{n+1}).
  \]
So, by induction on $w$, we get that
  \[
\tilde{q}^w_{u,v} \in\bz_+ [e^{-\beta}-1]_{\beta\in\Del^+} .
  \]

  (b) As a consequence of Theorem \ref{t.pn.2} and Corollary \ref{c.pn.3}, we get
  that in the non-equivariant $K$-theory $K(\Bbb P^n)$, we have
  \[p^w_{u,v}=0, \,\,\text{for}\, u+v>w\]
  and
  \[p^w_{u,w-u}=1, \,\,\text{for any} \,u\leq w; \,\,p^w_{u,w-u-1}=-1,
  \,\,\text{for any} \,u\leq w-1.\]
  \end{remark}

  \subsection{Structure constants with respect to the structure sheaf basis of $K_T(\Bbb P^n)$}
  We give explicit formulas for the structure constants with respect to the structure sheaf basis.
 These are strikingly similar to the formulas for the structure constants in the
  dual structure sheaf basis, but they differ subtly.  We will state these formulas here.  We omit most details of
  the proofs, which are very similar to the proofs in the previous subsection.

  Let $w_o$ denote the longest element of the Weyl group of $SL_{n+1}$, so
  $w_o(\veps_{i}) = \veps_{n+2 - i}$.  For $u \in [n]$, let $\bar{u} = n-u$.
  To state our formulas, it will be convenient to introduce the notation
  $r_{u,v}^w = w_o(b_{\bar{u},\bar{v}}^{\bar{w}})$.

    \begin{theorem} \label{t.pn.4} For any $0\leq u,v,w\leq n$,
  \begin{multline*}
(-1)^{u+v+w} r^w_{u,v}\\
 = e^{\chi_{w}-\chi_{u}-\chi_{v}} \biggl[
\frac{\bigl(\prod^u_{i=1} (1-te^{\eps_i})\bigr)\, \bigl(\prod^v_{i=1} (1-te^{\eps_i})\bigr)}
{\prod^{w+1}_{i=1} (1-te^{\eps_i})}\bigg]_{u+v-w} .\qquad\quad
  \end{multline*}
  \end{theorem}

  \begin{proof}
  Using Remark \ref{r.GR} and Proposition \ref{p.2.1}, we see that
    $$
  b_{\bar{u},\bar{v}}^{\bar{w}}= \chi  ( X\times X,\, D_*\xi^{\bar{w}}\otimes
([ \co_{X_{\bar{u}}}] \boxtimes [ \co_{X_{\bar{v}}}] )).
$$
Arguing as in the proof of Proposition \ref{p.2.1}, we see  that the right  side
of the above identity is equal
to
$$
\chi(X \times X, D_* [\co_{X_{\bar{u}}}] \otimes (\xi^{\bar{w}} \boxtimes [\co_{X_{\bar{v}}}] )).
$$
The case $w = 0$ of the theorem can be checked separately, so assume $w>0$,
i.e., $\bar{w} <n$.  Then, the above expression equals
$$
\chi(X_{\bar{u}} \cap X^{\bar{w}}, [\co_{X_{\bar{v}}}]) -
\chi(X_{\bar{u}} \cap X^{\bar{w}+1}, [\co_{X_{\bar{v}}}]) .
$$
This can be calculated as in the proof of Theorem \ref{t.pn.1}; we omit the details.
\end{proof}

Arguing as in the proof of Corollary \ref{c.pn.2} gives the following result.

\begin{corollary} \label{c.pn.5}
If $r^w_{u,v} \neq 0$, then $u,v\leq w\leq u+v$.
\end{corollary}

The next result gives ``initial condition" for the $r_{u,v}^w$.
\begin{proposition} \label{p.pn.initialr}
$r_{u,0}^w = \delta_{u,w}$.
\end{proposition}

\begin{proof}
This follows because $[\co_{X_n}]$ is the identity
element in $K_T(\Bbb P^n)$.  (Alternatively,
the proposition can be deduced from Theorem \ref{t.pn.4}.)
\end{proof}

Write $\tilde{r}_{u,v}^w = (-1)^{u+v+w}r_{u,v}^w$.

\begin{theorem} \label{t.pn.5} If $v \geq 1$, then
  \[
\tilde{r}^w_{u,v} =  ( e^{\eps_{u+1}-\eps_1}-1 )\,
\tilde{r}^{w-1}_{u-1,v-1}(\eps_2, \dots ,\eps_{n+1})
+ e^{\eps_{u+1}-\eps_1}\, \tilde{r}^{w-1}_{u,v-1}(\eps_2, \dots ,
\eps_{n+1}).
  \]
  \end{theorem}

  \begin{proof}
  This is similar to the proof of Theorem \ref{t.pn.2}; we omit the details.
  \end{proof}

  The above two  results imply that Conjecture \ref{conj.GR} holds for projective spaces
  (cf. Remark \ref{r.GR}).

  \begin{theorem} \label{t.pn.6}
   For all $0\leq u,v,w\leq n$,
  \[
\tilde{r}^w_{u,v} \in \bz_+ [ e^{-\al}-1 ]_{\al\in\Del^+}.
  \]
Hence,
 \[
(-1)^{n+u+v+w} b^w_{u,v} \in \bz_+ [ e^{\al}-1 ]_{\al\in\Del^+}.
  \]
  \end{theorem}

  \begin{proof}
  The first result follows from Proposition \ref{p.pn.initialr} and Theorem \ref{t.pn.5}.
  The second result follows from the first, since $w_o$ takes negative
  roots to positive roots.
  \end{proof}

  \section{A more general positivity conjecture} \label{s.general}
We revert to the notation and assumptions of Section 3.
The following conjecture is an equivariant generalization of
\cite[Theorem 1]{Bri:02}.  By Proposition \ref{p.pn.1},
this conjecture, with $G \times G$ in place of $G$ and $T'$ equal to the diagonal torus in $T \times T$,
would imply Conjecture \ref{conj.GK}.

    \begin{conjecture}  \label{conj.6.1} Let $T'$ be a subtorus of $T$ and let
    $Y\subset G/P$ be a $T'$-stable
irreducible subvariety with rational singularities.  Express, in $K_{T'}(G/P)$,
  \[
[\co_Y] = \sum_{w\in W^P} a^Y_w [\co_{X^P_w}] .
  \]
Then,
  \[
(-1)^{\codim Y +\codim X_w^P}\, a^Y_w\in\bz_+
[e^{-\beta}_{\mid T'}-1]_{\beta\in\Del^+} .
  \]
  \end{conjecture}

  \begin{remark}  \begin{enumerate}
  \item[1)] By (a subsequent) Proposition \ref{p.6.4} and Remark \ref{7.7}(a), the above conjecture
is true for $Y = X^v_P\subset G/P$.
  \item[2)] We have verified the above conjecture by an explicit
calculation for $G=SL_3$, $P=B$ and $Y=X_w\cap X^v$ for any $v,w\in
W$.
  \end{enumerate}
  \end{remark}

  In the next proposition we view $\bp^1$ as having projective
  coordinates $[x_0:x_1]$, so $\frac{x_0}{x_1}$ is a rational function
  on $\bp^1$.  We write $0 = [0:1]$ and $\infty = [1:0]$.

    \begin{proposition} \label{p.alpha}
    Suppose $T$ acts on $\bp^1$ such that $0$ and $\infty$ are $T$-fixed and $\frac{x_0}{x_1}$ is a
    $T$-weight vector with weight $-\al$.
  Let $X$ be an irreducible $T$-variety and $\phi : X \to \bp^1$ a
dominant $T$-equivariant morphism.  Then, in $K_T(X)$,
  \[
[\co_{\phi^{-1}(\infty )}] = (1-e^{\al})[\co_X] +
e^{\al}[\co_{\phi^{-1}(0)}].
  \]
  \end{proposition}

\begin{proof}
Since $\phi$ is a flat morphism (see \cite[Ch.~III, Prop.~9.7]{Har:77}), $\phi^*[\co_{\{0\}}] =
[\co_{\phi^{-1}(0)} ]$ and
similarly for $\co_{\{\infty\}}$.  Hence, it suffices to show that on $\bp^1$,
  \[
  [\co_{\{\infty\}}] = (1-e^{\al})[\co_{\bp^1}] + e^{\al}[\co_{\{
0\}}],
  \]
since applying $\phi^*$ gives the desired equation.  We have
exact sequences
  \[
0 \to \ci_0 \to \co_{\bp^1} \to \co_{\{ 0\}} \to 0,
\]
and
\[ 0\to
\ci_{\infty} \to \co_{\bp^1} \to \co_{\{\infty\}} \to 0 ,
  \]
  where $\ci_0$ and $\ci_{\infty}$ are the ideal sheaves
  of $\{ 0 \}$ and $\{ \infty \}$, respectively.
Nonequivariantly, $\ci_0 = \ci_{\infty} = \co_{\bp^1} (-1)$,
so $\ci_0
\otimes \ci_{\infty}^*$ is non-equivariantly isomorphic to
$\co_{\bp^1}$.   Near $0=[0: 1]$ the sheaf
$\ci_0$ is generated by $x_0/x_1$, which has weight $-\al$, and $\ci_{\infty}$
near $0$ is generated by $1$.  Hence, as $T$-equivariant sheaves,
$\ci_0\simeq e^{-\al}\ci_{\infty}$.  So,
  \begin{align*}
[\co_{\{\infty\}}] &= [\co_{\bp^1}] - [\ci_{\{\infty\}}] = [\co_{\bp^1}]
- e^{\al}[\ci_0] = [\co_{\bp^1}] - e^{\al}  ( [\co_{\bp^1}] -
[\co_{\{0\}}] ) \\
&= (1-e^{\al})[\co_{\bp^1}] + e^{\al}[\co_{\{0\}}]
  \end{align*}
as desired.
 \end{proof}

  Let $w\in W$. If $s$ is a simple reflection with $sw<w$, then $sX_w = X_w$ and hence
$[ \co_{sX_w}] = [\co_{X_w}]$.  On the other hand, if $sw>w$ then
we have the following result.

  \begin{proposition} \label{p.6.3}  If $s$ is a simple reflection
  with $sw > w$, then
  \[
[\co_{sX_w}] = e^{-\al} [\co_{X_w}] - (e^{-\al}-1) [\co_{X_{sw}} ],
  \]
where $\al$ is the simple root corresponding to $s$.
  \end{proposition}

  \begin{proof}  Let $P_s$ be the minimal parabolic
  corresponding to $s$.  Consider
  \begin{align*}
P_s &\times^B X_w \overset{\mu}\twoheadrightarrow X_{sw}  \\
&\downarrow \pi \\
P_s &\slash B = \bp^1 ,
  \end{align*}
  where $\mu$ takes the $B$-orbit $[p,x] \mapsto px$ and $\pi$ takes
  $[p,x]\mapsto p\,\text{mod}\,B.$
Then,
  \begin{align*}
\pi^{-1}(0) &= \{ 1 \}\times X_w, \; \pi^{-1}(\infty ) = \{ s \} \times X_w .\\
\intertext{So,  by Proposition \ref{p.alpha},} [\co_{\{ s \} \times X_w}] &= e^{-\al} [\co_{ \{1 \} \times X_w}] +
(1-e^{-\al})[\co_{P_s\times^B X_w}] .
  \end{align*}
Push forward the above identity to $X_{sw}$ via $\mu$ to get the result. (Here we have used
\cite[Proposition 3.2.1]{BrKu:05}.)
  \end{proof}

  \begin{lemma}  \label{l.6.4} For any $T$-stable closed subscheme $Y\subset G/P$,
  write in $K_T(G/P)$,
    \beqn \label{e.6.4a}
[\co_Y] = \sum_{w\in W^P}\, P_w[\co_{X_w^P}], \,\text{ for some (unique)} \,\, P_w\in R(T).
  \eeqn
Then, for any $v\in W$,
  \[
[\co_{v^{-1}Y}] = \sum_{w\in W^P}\, (v^{-1} P_w)[\co_{v^{-1}X_w^P}] .
  \]
  \end{lemma}

  \begin{proof}  Let $f: T \to T'$ be any homomorphism.  If $X$ is any scheme with $T'$-action,
  and $T$ acts on $X$ via $f$, then there is a map
 $f^*: K_{T'}(X) \to K_T(X)$ extending the natural pull-back map $R(T') \to R(T)$.
 For any $T'$-stable closed subscheme $Y$ of $X$,
 $f^*$ takes the class of $\co_Y$ in $K_{T'}(X)$ to the class of $\co_Y$ in $K_T(X)$.
 We now apply this to $X = G/P$ and $f: T \to T$ given by $f(t) = vtv^{-1}$.
 Since $f^*r = v^{-1} r$ for $r \in R(T)$, we get from \eqref{e.6.4a} the equation
 \beqn \label{e.6.4b}
 [\co_{Y}] = \sum_{w\in W^P}\, (v^{-1} P_w)[\co_{X_w^P}],
 \eeqn
 where in this equation $T$ is viewed as acting on $G/P$ through $f$.  Write
 $(G/P, \odot)$ to indicate $G/P$ with this new action of $T$.

 Consider the automorphism
\[\phi_v : G/P \to (G/P, \odot ),\,\,
gP \mapsto \dot{v}gP,
  \]
  where $\dot{v}$ is a representative of $v$ in $N(T)$.
This is $T$-equivariant, where $T$ acts on the source $G/P$
by the standard action.  Then,
 $\phi^*_v [\co_{X_w^P}] = [\co_{v^{-1} X_w^P}] $ and
 $\phi^*_v [\co_{Y}] = [\co_{v^{-1} Y}] $.
Since $\phi^*_v$ is $R(T)$-linear, applying $\phi^*_v$ to
\eqref{e.6.4b} proves the result.
\end{proof}

  \begin{proposition}  \label{p.6.4} Write $[\co_{X^w}] = \sum_u\,
  e_{w,u}[\co_{X_u}]$.  Then,
    \[
(-1)^{\codim X^w + \codim X_u}  e_{w,u} \in
\bz_+[e^{-\beta}-1]_{\beta\in\Del^+} .
  \]
  \end{proposition}

  \begin{proof}   For any $v,w\in W$,  write
  \[
[\co_{vX_w}] = \sum\, f^v_{w,u} [ \co_{X_u}].\]
  We prove by induction on $\ell (v)$, that for any $u,w\in W$,
\beqn \label{e.6.4}(-1)^{\codim X_w +\codim X_u}\, f^v_{w,u}
\in\bz_+[e^{-\beta}-1]_{\beta\in\Del^+}.\eeqn
Of course, \eqref{e.6.4} is true
for $v=e$.  Now take $vs_i$ with $\ell (vs_i) > \ell (v)$.
If $s_iw<w$, then $[\co_{vX_w}]  = [\co_{v s_i X_w}] $ and
we are done.
If $s_iw>w$, then, by
Proposition \ref{p.6.3},
  \[
  [\co_{s_iX_w}] = e^{-\al_i}[\co_{X_w}] -  (
  e^{-\al_i}-1 ) [\co_{X_{s_i w}}] .
   \]
Thus, by Lemma \ref{l.6.4},
  \[
  [\co_{vs_iX_w}] = e^{-v\al_i}[\co_{vX_w}] -  (
  e^{-v\al_i}-1 ) [\co_{vX_{s_i w }}] .
   \]
Since $vs_i >v$, $v\al_i \in\Del^+$.  Moreover, by induction, for any $u\in W$,
$(-1)^{\codim X_w +\codim X_u}\, f^v_{w,u}$ and
$(-1)^{\codim X_w-1 +\codim X_u}\, f^v_{s_iw,u}$
 are in $\bz_+[e^{-\beta}-1]_{\beta\in\Del^+}$.
Hence, $(-1)^{\codim X_w +\codim X_u}\, f^{vs_i}_{w,u}
\in\bz_+[e^{-\beta}-1]_{\beta\in\Del^+}$.  This completes the induction and hence
\eqref{e.6.4} is proved for any $u,v,w\in W$. Since $X^w=w_o X_{w_o w}$,
the proposition follows.
  \end{proof}

  \begin{remark}  \label{7.7}
  (a) For any standard parabolic $P$ and any closed $T$-stable subvariety $Z\subset G/P$,
  since $\pi^*[\co_Z]=[\co_{\pi^{-1}(Z)}]$ (cf. the proof of Proposition \ref{p=b}),
  where $\pi:G/B \to G/P$ is the standard projection, the above proposition and
  \eqref{e.6.4}  remain true for the Schubert varieties in $G/P$.

  (b) Since any $T$-stable closed irreducible subvariety
of $\bp^n$ (under the standard action of the maximal torus $T$ of
$SL(n+1)$) is a $W$-translate of a Schubert variety of $\bp^n$,
Conjecture \ref{conj.6.1} is true for any $T$-stable closed irreducible
subvariety of $\bp^n$ (by virtue of \eqref{e.6.4}).

  \end{remark}

\def\cprime{$'$}
\providecommand{\bysame}{\leavevmode\hbox to3em{\hrulefill}\thinspace}
\providecommand{\MR}{\relax\ifhmode\unskip\space\fi MR }
\providecommand{\MRhref}[2]{%
  \href{http://www.ams.org/mathscinet-getitem?mr=#1}{#2}
}
\providecommand{\href}[2]{#2}

\vskip5ex

Addresses:

W.G.:  Department of Mathematics, University of Georgia, Athens,
 GA 30602-7403, USA

S.K.: Department of Mathematics, University of North Carolina,
 Chapel Hill, NC 27599-3250, USA

 \end{document}